\documentclass[11pt]{article}
\usepackage[english]{babel}
\usepackage[utf8]{inputenc}
\usepackage{fullpage}
\usepackage{enumerate}
\usepackage[toc]{appendix}

\usepackage{amsmath}
\usepackage{amsfonts}
\usepackage{amsthm}
\usepackage{mathtools}
\usepackage{dsfont}

\usepackage{graphicx}
\usepackage[colorlinks=true,linkcolor=black]{hyperref}

\usepackage[hang]{subfigure}

\newcommand{\R}{\mathbb R}
\newcommand{\N}{\mathbb N}

\newcommand{\cH}{\mathcal{H}}
\newcommand{\cK}{\mathcal{K}}

\newcommand{\cE}{\mathcal{E}}

\newcommand{\cL}{\mathcal{L}}
\newcommand{\cM}{\mathcal{M}}

\def\xC{{\rm C}}
\def\xLip{ {\rm Lip} }
\def\xlip{ {\rm lip} }

\def\xL{{\rm L}}
\def\xdiv{{\rm div}}

\newcommand{\V}{\|V\|}

\newcommand{\G}{G_{d,n}}
\newcommand{\Div}{\text{div}}

\newtheorem{prop}{Proposition}

\newtheorem{lemma}{Lemma}
\newtheorem{theo}{Theorem}

\newtheorem{dfn}{Definition}

\theoremstyle{remark}

\newtheorem{remk}{Remark}

\newtheorem{xmpl}{Example}

\DeclareMathOperator*{\argmin}{arg\,min}
\DeclareMathOperator*{\supp}{supp}

\DeclareMathOperator*{\heightex}{heightex}

\DeclareMathOperator*{\diam}{diam}

\title{Quantitative conditions of rectifiability for varifolds}
\author{Blanche \textsc{Buet}}

\begin{document}
     \maketitle
     
\newcommand{\one}{\mathds{1}}
\renewcommand{\phi}{\varphi}
\renewcommand{\epsilon}{\varepsilon}

\begin{abstract}
Our purpose is to state quantitative conditions ensuring the rectifiability of a $d$--varifold $V$ obtained as the limit of a sequence of $d$--varifolds $(V_i)_i$ which need not to be rectifiable. More specifically, we introduce a sequence $\left\lbrace \cE_i \right\rbrace_i$ of functionals defined on $d$--varifolds, such that if $\displaystyle \sup_i \cE_i (V_i) < +\infty$ and $V_i$ satisfies a uniform density estimate at some scale $\beta_i$, then $V = \lim_i V_i$ is $d$--rectifiable.

\noindent The main motivation of this work is to set up a theoretical framework where curves, surfaces, or even more general $d$--rectifiable sets minimizing geometrical functionals (like the length for curves or the area for surfaces), can be approximated by ``discrete'' objects (volumetric approximations, pixelizations, point clouds etc.) minimizing some suitable ``discrete'' functionals.  
\end{abstract}

	\tableofcontents

\section*{Introduction}

The set of regular surfaces lacks compactness properties (for Hausdorff convergence for instance), which is a problem when minimizing geometric energies defined on surfaces. In order to gain compactness, the set of surfaces can be extended to the set of varifolds and endowed with a notion of convergence (weak--$\ast$ convergence of Radon measures). Nevertheless, the problem turns to be the following: how to ensure that a weak--$\ast$ limit of varifolds is regular (at least in the weak sense of rectifiability)? W. K. Allard (see \cite{allard}) answered this question in the case where the weak--$\ast$ converging sequence is made of weakly regular surfaces (rectifiable varifolds to be precise). But what about the case when the weak--$\ast$ converging sequence is made of more general varifolds? Assume that we have a sequence of volumetric approximations of some set $M$, how can we know if $M$ is regular ($d$--rectifiable for some $d$), knowing only its successive approximations ?

\begin{minipage}{0.30\textwidth}
\includegraphics[width=0.95\textwidth]{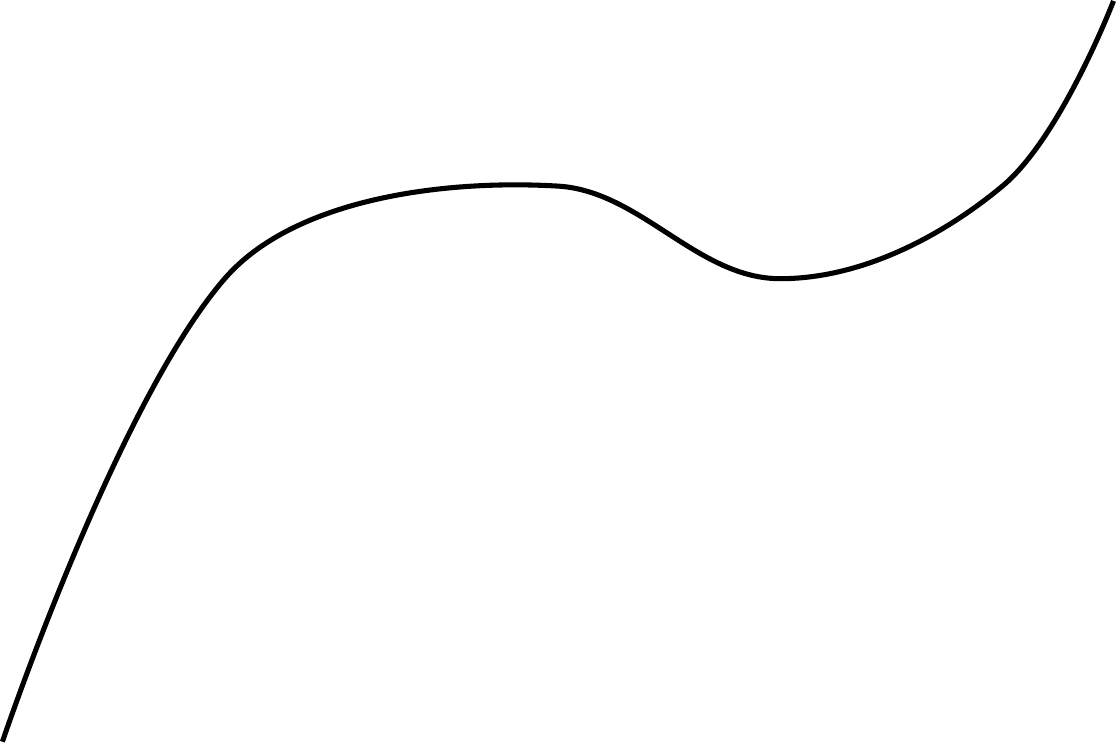}
\end{minipage}
\begin{minipage}{0.30\textwidth}
\includegraphics[width=0.95\textwidth]{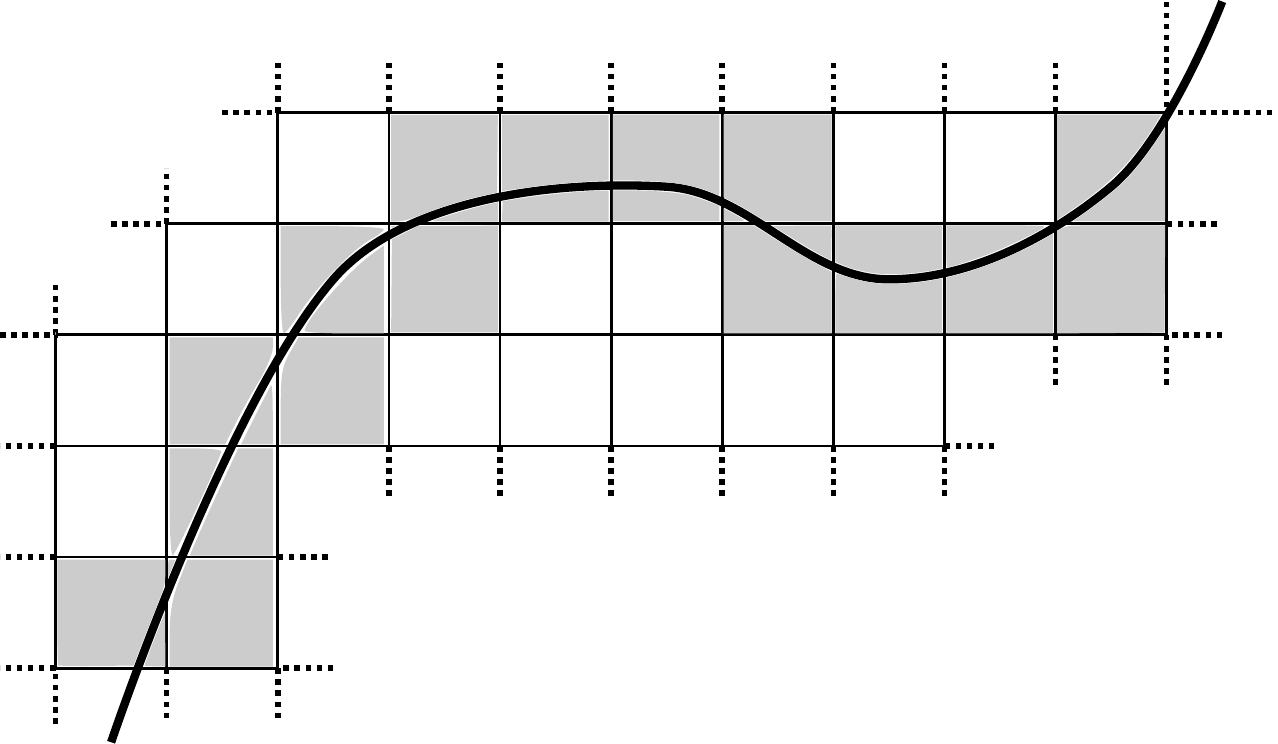}
\end{minipage}
\begin{minipage}{0.30\textwidth}
\includegraphics[width=0.95\textwidth]{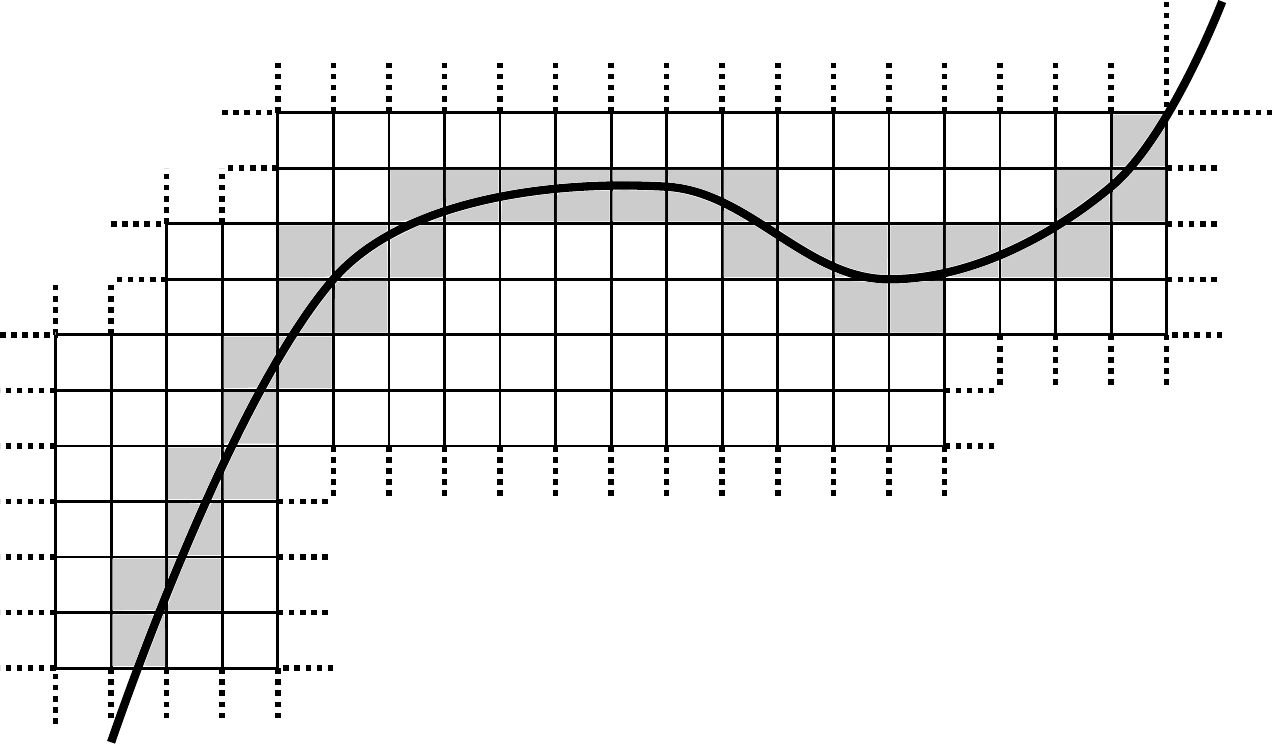}
\end{minipage}

As a set and its volumetric approximations can be endowed with a structure of varifold (as we will see), this problem can be formulated in terms of varifolds: we are interested in quantitative conditions on a given sequence of $d$--varifolds ensuring that the limit (when it exists) is rectifiable. Before going into technical details, let us consider the problem of rectifiability in simplified settings.
\begin{itemize}
\item First, let $f : \R \rightarrow \R$. We are look for conditions ensuring that $f$ is differentiable (in some sense). The most simple answer is to impose that the difference quotient has a finite limit everywhere. But assume that moreover, we ask for something more quantitative, that is to say some condition that could be expressed through bounds on some well chosen quantities (for instance, from a numerical point of view, it is easier to deal with bounded quantities than with the existence of a limit)). We will refer to this kind of condition as ``quantitative conditions'' (see also \cite{david_semmes2}). There exists an answer by Dorronsoro \cite{dorronsoro} (we give here a simplified version, see \cite{david_semmes}).
\begin{theo}[see \cite{dorronsoro} and \cite{david_semmes}]
Let $f: \R^d \rightarrow \R$ be locally integrable and let $q \geq 1$ such that $\displaystyle q < \frac{2d}{d-2}$ if $d>1$. Then, the distributional gradient of $f$ is in $\xL^2$ if and only if
\[
\int_{\R^d} \int_0^1 \gamma_q (x,r)^2 \, \frac{dr}{r} \, dx < + \infty \quad \text{with} \quad \gamma_q (x,r)^q = \inf_{\substack{a \text{ affine} \\ \text{function}}} \frac{1}{r^{d+1}} \int_{B_r(x)} \left| f(y) - a(y) \right|^q \, dy 
\]
\end{theo}
The function $\gamma_q$ penalizes the distance from $f$ to its best affine approximation locally everywhere. This theorem characterizes the weak differentiability (in the sense of a $\xL^2$ gradient) quantitatively in terms of $xL^2$--estimate on $\gamma_q$ (with the singular weight $\frac{1}{r}$). 
\item Now, we take a set $M$ in $\R^n$ and we ask the same question: how to ensure that this set is regular (meaning $d$--rectifiable for some $d$)? Of course, we are still looking for quantitative conditions. This problem has been studied by P.W. Jones (for $1$--rectifiable sets) in connection with the \emph{travelling salesman problem} (\cite{jones1990rectifiable}) then by K. Okikiolu (\cite{okikiolu1992characterization}), by S. Semmes and G. David (\cite{david1991singular}) and by H. Pajot (\cite{MR1459297}). As one can see in the following result stated by H. Pajot in \cite{MR1459297}, the exhibited conditions are not dissimilar to Dorronsoro's. We first introduce the $\xL^q$ generalization of the so called Jones' $\beta$ numbers, (see \cite{jones1990rectifiable} for Jones' $\beta$ numbers and \cite{MR1459297} for the $\xL^q$ generalization):
\begin{dfn}
Let $M \subset \R^n$ and $d \in \N$, $d \leq n$.
\begin{align*}
\beta_\infty (x,r,M) & =  \inf_{P \text{ affine } d-\text{plane}} \sup_{y \in M \cap B_r(x)} \frac{d(y,P)}{r} \quad & \text{if } B_r(x) \cap M \neq \emptyset \: , \\
\beta_\infty (x,r,M)& =  0 \quad & \text{if } B_r(x) \cap M = \emptyset \: , \\
\beta_q (x,r,M) & =  \inf_{P \text{ affine } d-\text{plane}} \left( \frac{1}{r^d} \int_{y \in B_r(x) \cap M} \left( \frac{d(y,P)}{r} \right)^q \, d \cH^d(y) \right)^{\frac{1}{q}} \quad & \text{if } 1 \leq q < +\infty \: .
\end{align*}
\end{dfn}
The $\beta_q (x,r,M)$ measure the distance from the set $M$ to its best affine approximation at a given point $x$ and a given scale $r$. 
\begin{theo}[\cite{MR1459297}]\label{Pajot}
Let $M \subset \R^n$ compact with $\cH^d (M) < +\infty $. Let $q$ be such that 
\[
\left\lbrace \begin{array}{lll}
1 \leq q \leq \infty & \text{if} & d=1 \\
1 \leq q < \displaystyle \frac{2d}{d-2} & \text{if} & d \geq 2 \: .
\end{array} \right.
\]
We assume that for $\cH^d$--almost every $x \in M$, the following properties hold:
\begin{enumerate}[(i)]
\item $\theta_\ast^d (x, M) = \displaystyle \liminf_{r \downarrow 0} \displaystyle \frac{\cH^d (M \cap B_r(x))}{\omega_d r^d} > 0 $,
\item $\displaystyle \int_{r=0}^1 \beta_q (x,r,M)^2 \, \frac{dr}{r} < \infty $.
\end{enumerate}
Then $M$ is $d$--rectifiable.
\end{theo}
\item Let us get closer to our initial question: now we consider the same question in the context of varifolds. Recall that from a mathematical point of view, a $d$--varifold $V$ in $\Omega \subset \R^n$ is a Radon measure on the product $\Omega \times \G$, where
\[
\G = \left\lbrace d\text{--dimensional subspaces of } \R^n \:  \right\rbrace \: .
\] Varifolds can be loosely seen as a set of generalized surfaces: let $M$ be a $d$--submanifold (or a $d$--rectifiable set) in $\Omega$ and denote by $T_x M$ its tangent plane at $x$, then the Radon measure $V(x,P) = \cH^d_{| M}(x) \otimes \delta_{T_x M}(P)$ is a $d$--varifold associated to $M$, involving both spatial and tangential information on $M$. The measure obtained by projecting $V$ on the spatial part $\Omega$ is called the mass $\V$. In the previous specific case where $V$ comes from a $d$--rectifiable set $M$ then the mass is $\V = \cH^d_{| M}$. See the next section for more details about varifolds. We can now state the first result that we obtain in this paper about quantitative conditions of rectifiability in the context of varifolds:
\begin{theo} \label{static_thm}
Let $\Omega \subset \R^n$ be an open set and let $V$ be a $d$--varifold in $\Omega$ with finite mass $\V (\Omega) < +\infty$. Assume that:
\begin{enumerate}[(i)]
\item there exist $0 < C_1 < C_2$ such that for $\V$--almost every $x \in \Omega$ and for every $r>0$,
\begin{equation} \label{static_density_condition}
C_1 r^d \leq \| V \| (B_r(x)) \leq C_2 r^d \: ,
\end{equation}
\item $\displaystyle \int_{\Omega \times \G} E_0 (x,P,V) \, d V(x,P) < +\infty $, where 
\[ E_0 (x,P,V) = \int_{r=0}^1 \frac{1}{r^d} \int_{y \in B_r(x) \cap \Omega} \left( \frac{d(y-x,P)}{r} \right)^2 \, d \V(y) \, \frac{dr}{r} 
\]
defines the \emph{averaged height excess}.
\end{enumerate}
Then $V$ is a rectifiable $d$--varifold.
\end{theo}
The first assumption is called Ahlfors-regularity. It implies in particular that $V$ is $d$--dimensional but with some uniform control on the $d$--density. 
Adding the second assumption both ensures that the support $M$ of the mas measure $\V$ is a $d$--rectifiable set and that the tangential part of $V$ is coherent with $M$, that is to say $V = \V \otimes \delta_{T_x M}$. We will refer to these two conditions as \emph{static} quantitative conditions of rectifiability for a given $d$--varifold, by opposition to the next conditions, involving the limit of a sequence of $d$--varifolds, which we will refer to as the \emph{approximation} case. These static conditions are not very difficult to derive from Pajot's theorem, the difficult part is the next one: the approximation case.
\item Now we consider a sequence $(V_i)_i$ of $d$--varifolds (weakly--$\ast$) converging to a $d$--varifold $V$.The problem is to find quantitative conditions on $(V_i)_i$ that ensure the rectifiability of $V$? The idea is to consider the static conditions with uniform bounds and using a notion of scale encoded by the parameters $\alpha_i$ and $\beta_i$ in the following result:
\begin{theo} \label{dynamic_thm}
Let $\Omega \subset \R^n$ be an open set and let $(V_i)_i$ be a sequence of $d$--varifolds in $\Omega$ weakly--$\ast$ converging to some $d$--varifold $V$ of finite mass $\V (\Omega) < +\infty$.  Fix two decreasing and infinitesimal (tending to $0$) sequences of positive numbers $(\alpha_i)_i$ and $(\beta_i)_i$ and assume that:
\begin{enumerate}[(i)]
\item there exist $0 < C_1 < C_2$ such that for $\| V_i \|$--almost every $x \in \Omega$ and for every $\beta_i < r < d(x,\Omega^c)$,
\[
C_1  r^d \leq \| V_i \| (B_r(x)) \leq C_2 r^d \:  \label{(H1)},
\]
\item $\displaystyle \sup_i \int_{\Omega \times \G} E_{\alpha_i} (x,P,V_i) \, d V_i(x,P) < +\infty $, where
\[
E_{\alpha} (x,P,W) = \int_{r=\alpha_i}^1 \frac{1}{r^d} \int_{y \in B_r(x) \cap \Omega} \left( \frac{d(y-x,P)}{r} \right)^2 \, d \|W\|(y) \, \frac{dr}{r}
\]
denotes the $\alpha$--\emph{approximate averaged height excess}.
\end{enumerate}
Then $V$ is a rectifiable $d$--varifold.
\end{theo}

\end{itemize}

We stress that the sequence $(V_i)_i$ in Theorem \ref{dynamic_thm} is not necessarily made of rectifiable $d$--varifolds.
The parameters $\alpha_i$ and $\beta_i$ allow to study the varifolds at a large scale (from far away). The main difficulty in the proof of Theorem \ref{dynamic_thm} is to understand the link between 
\begin{enumerate}[$-$]
\item the choice of $\alpha_i$ ensuring a good convergence of the successive approximate averaged height excess energies $E_{\alpha_i} (x,P,V_i)$ to the averaged height excess energy $E_0 (x,P,V)$
\item and a notion of convergence speed of the sequence $(V_i)_i$ obtained thanks to a strong characterization of weak--$\ast$ convergence.
\end{enumerate}
\noindent In the following example, we can guess that the parameters $\alpha_i$ and $\beta_i$ must be large with respect to the size of the mesh. Loosely speaking, in figure $(a)$, even in the smallest ball, the grey approximation ``looks'' $1$--dimensional. On the contrary, if we continue zooming like in figure $(b)$, the grey approximation ``looks'' $2$--dimensional. The issue is to give a correct sense to this intuitive fact.
\setcounter{subfigure}{0}
\begin{figure}[!htp]
\subfigure[]{\includegraphics[width=0.60\textwidth]{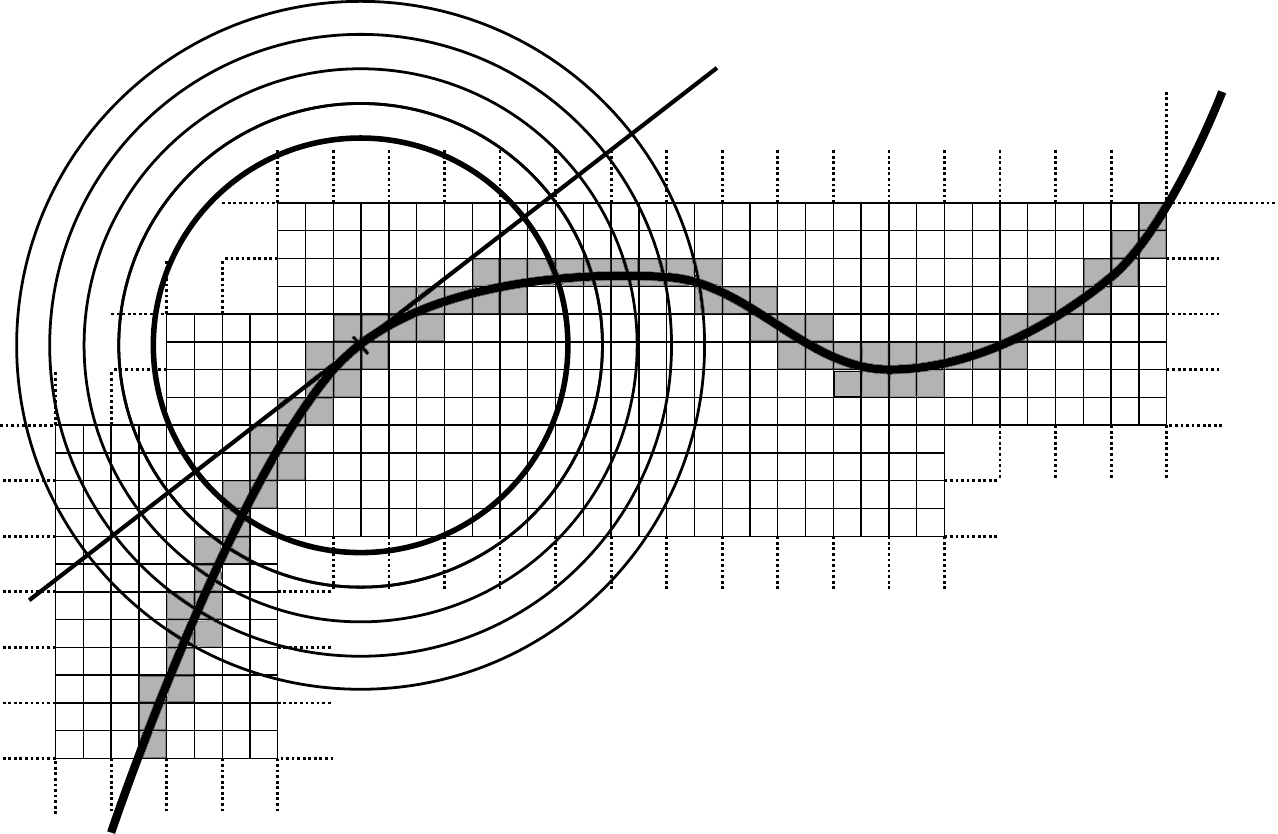}}
\subfigure[]{\includegraphics[width=0.38\textwidth]{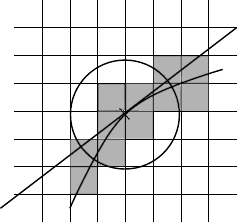}}
\end{figure}

The plan of the paper is the following: in section \ref{generalities_rectif} we collect some basic facts about rectifiability and varifolds that we need thereafter. Then in section \ref{static_section}, we state and prove quantitative conditions of rectifiability for varifolds in the static case. In section \ref{approximation_section}, we first establish a result of uniform convergence for the pointwise averaged height excess energies $E_\alpha$ thanks to a strong characterization of weak--$\ast$ convergence. This allows us to state and prove quantitative conditions of rectifiability for varifolds in the approximation case. In the appendix, we consider some sequence of $d$--varifolds weakly--$\ast$ converging to some rectifiable $d$--varifold $V= \theta \cH^d_{| M} \otimes \delta_{T_x M}$ (for some $d$--rectifiable set $M$) and we make a connection between the minimizers of $E_{\alpha_i} (x, \cdot , V_i)$, with respect to $P \in \G$, and the tangent plane $T_x M$ to $M$ at $x$.


\section{Some facts about rectifiability and varifolds} \label{generalities_rectif}

This section contains basic definitions and facts about rectifiability and varifolds. We start by fixing some notations.

From now on, we fix $d, \, n \in \N$ with $1\leq d < n$ and an open set $\Omega \subset \R^n$. Then we adopt the following notations.
\begin{enumerate}[$-$]
\vspace{-0.25cm}
\item $\cL^n$ is the $n$--dimensional Lebesgue measure.
\vspace{-0.25cm}
\item $\cH^d$ is the $d$--dimensional Hausdorff measure.
\vspace{-0.25cm}
\item $\xC_c^k (\Omega)$ is the space of continuous compactly supported functions of class $\xC^ k$ in $\Omega$.
\vspace{-0.25cm}
\item $B_r(x) = \left\lbrace y \, | \, |y-x| < r \right\rbrace$ is the open ball of center $x$ and radius $r$.
\vspace{-0.25cm}
\item $\G = \left\lbrace P \subset \R^n \, | \, P \text{ is a vector subspace of dimension } d \right\rbrace$.
\vspace{-0.25cm}
\item $A \triangle B = (A \cup B) \setminus (A \cap B )$ is the symmetric difference.
\vspace{-0.25cm}
\item $\xLip_k (\Omega)$ is the space of Lipschitz functions in $\Omega$ with Lipschitz constant less or equal to $k$.
\vspace{-0.25cm}
\item $\omega_d = \cL^d (B_1(0))$ is the $d$--volume of the unit ball in $\R^d$.
\vspace{-0.25cm}
\item For $P \in \G$, $\Pi_P$ is the orthogonal projection onto $P$.
\vspace{-0.25cm}
\item Let $\omega$ and $\Omega$ be two open sets then $\omega \subset\subset \Omega$ means that $\omega$ is relatively compact in $\Omega$.
\vspace{-0.25cm}
\item Let $\mu$ be a measure in some measurable topological space, then $\supp \mu$ denotes the topological support of $\mu$.
\vspace{-0.25cm}
\item Let $A \subset \Omega$ then $A^c = \Omega \setminus A$ denotes the complementary of $A$ in $\Omega$.
\vspace{-0.25cm}
\item Given a measure $\mu$, we denote by $| \mu |$ its total variation.
\vspace{-0.25cm}
\end{enumerate}

\subsection{Radon measure and weak--$\ast$ convergence}

We recall here some useful properties concerning vector-valued Radon measures and weak--$\ast$ convergence. See \cite{evans} and \cite{ambrosio} for more details.

\begin{dfn}[weak--$\ast$ convergence of Radon measures, see. \cite{ambrosio} def. 1.58 p. 26]
Let $\mu$ and $( \mu_i )_i$ be $\R^m$--vector valued Radon measures in $\Omega \subset \R^n$. We say that $\mu_i$ weakly--$\ast$ converges to $\mu$, denoted $\mu_i \xrightharpoonup[i\to \infty]{\ast} \mu$ if for every $\phi \in \xC_c (\Omega , \R^m)$,
\[
\int_\Omega \varphi \cdot d \mu_i \xrightarrow[i \to \infty]{} \int_\Omega \varphi \cdot d \mu \: .
\]
\end{dfn}

\noindent Thanks to Banach-Alaoglu weak compactness theorem, we have the following result in the space of Radon measures.

\begin{prop}[Weak--$\ast$ compactness, see \cite{ambrosio} Theorem. 1.59 and 1.60 p. 26]
Let $\left( \mu_i \right)_i$ be a sequence of Radon measures in some open set $\Omega \subset \R^n$ such that $\sup_i | \mu_i | (\Omega ) < \infty$ then there exist a finite Radon measure $\mu$ and a subsequence $(\mu_{\phi(i)})_i$ weakly--$\ast$ converging to $\mu$.
\end{prop}

Let us now study the consequences of weak--$\ast$ convergence on Borel sets.

\begin{prop}[see $1.9$ p.$54$ in \cite{evans}] \label{weak_cv_borel_sets}
Let $\left( \mu_i \right)_i$ be a sequence of positive Radon measures weakly--$\ast$ converging to $\mu$ in some open set $\Omega \subset \R^n$. Then,
\begin{enumerate}
\item for every compact set $K \subset \Omega$, $\limsup_i \mu_i (K) \leq \mu (K)$ and for every open set $U \subset \Omega$, $\mu (U) \leq \liminf_i \mu_i (U)$.
\item $\lim_i \mu_i (B) = \mu (B)$ for every Borel set $B \subset \Omega$ such that $\mu (\partial B)=0$.
\end{enumerate}
\end{prop}

\noindent Each one of the two properties in Proposition \ref{weak_cv_borel_sets} is actually a characterization of weak--$\ast$ convergence. Let us state a similar result in the vector case.

\begin{prop}[see \cite{ambrosio} Prop. 1.62(b) p. 27]
Let $\Omega \subset \R^n$ be an open set and let $( \mu_i )_i$ be a sequence of $\R^m$--vector valued Radon measures weakly--$\ast$ converging to $\mu$. Assume in addition that the total variations $| \mu_i |$ weakly--$\ast$ converge to some positive Radon measure $\lambda$. Then $| \mu | \leq \lambda$ and for every Borel set $B \subset \Omega$ such that $\lambda ( \partial B )=0$, $\mu_i (B) \rightarrow \mu(B)$. More generally,
\[
\int_\Omega u \cdot d \mu_i \longrightarrow \int_\Omega u \cdot d \mu 
\] for every measurable bounded function $u$ whose discontinuity set has zero $\lambda$--measure.
\end{prop}

We end this part with a result saying that, for a given Radon measure $\mu$, among all balls centred at a fixed point, at most a countable number of them have a boundary with non zero $\mu$--measure.

\begin{prop} \label{boundary_weight}
Let $\mu$ be a Radon measure in some open set $\Omega \subset \R^n$. Then,
\begin{enumerate}[(i)]
\item For a given $x \in \Omega$, the set of $r \in \R_+$ such that $\mu (\partial B_r (x)) > 0$ is at most countable. In particular,
\[
\cL^1 \{ r \in \R_+ \: | \: \mu (\partial B_r (x) \cap \Omega) > 0 \} = 0 \: .
\]
\item For almost every $r \in \R_+$,
\[
\mu \left\lbrace x \in \Omega \: | \: \mu (\partial B_r (x) \cap \Omega) > 0 \right\rbrace = 0 \: .
\]
\end{enumerate}
\end{prop}

\begin{proof}
The first point is a classical property of Radon measures and comes from the fact that monotone functions have at most a countable set of discontinuities, applied to $r \mapsto \mu (B_r(x))$.
For the second point, we use Fubini Theorem to get
\begin{align*}
\int_{r \in \R_+} \mu \left\lbrace x \in \Omega \: | \: \mu (\partial B_r (x) \cap \Omega) > 0 \right\rbrace \, dr & = \int_{x \in \Omega} \int_{r \in \R_+} \mathds{1}_{ \{ (x,r) \: | \: \mu (\partial B_r (x) \cap \Omega) > 0 \} } (x,r) \, d \mu(x) \, dr \\
& = \int_{x \in \Omega} \cL^1 \{ r \in \R_+ \: | \: \mu (\partial B_r (x) \cap \Omega) > 0 \} \, d \mu(x) =0 \: ,
\end{align*}
thanks to $(i)$.
\end{proof}

\noindent These basic results will be widely used throughout this paper.

\subsection{Rectifiability and approximate tangent space}

\begin{dfn}[$d$--rectifiable sets, see definition $2.57$ p.$80$ in \cite{ambrosio}]
Let $M \subset \R^n$. $M$ is said to be \emph{countably $d$--rectifiable} if there exist countably many Lipschitz functions $f_i : \R^d \rightarrow \R^n$ such that
\[
M \subset M_0 \cup \bigcup_{i \in \N} f_i (\R^d) \text{ with } \cH^d(M_0) = 0 \: .
\]
If in addition $\cH^d (M) < +\infty$ then $M$ is said \emph{$d$--rectifiable}.
\end{dfn}
\noindent Actually, it is equivalent to require that $M$ can be covered by countably many Lipschitz $d$--graphs up to a $\cH^d$--negligible set and thanks to Whitney extension theorem, one can ask for $\xC^1$ $d$--graphs. We can now define rectifiability for measures.

\begin{dfn}[$d$--rectifiable measures, see definition $2.59$ p.$81$ in \cite{ambrosio}]
Let $\mu$ be a positive Radon measure in $\R^n$. We say that $\mu$ is $d$--rectifiable if there exist a countably $d$--rectifiable set $M$ and a Borel positive function $\theta$ such that $\mu = \theta \cH^d_{| M}$.
\end{dfn}

\noindent Thus, a set $M$ is countably $d$--rectifiable if and only if $\cH^d_{| M}$ is a $d$--rectifiable measure. When blowing up at a point, rectifiable measures have the property of concentrating on affine planes (at almost any point). This property leads to a characterization of rectifiable measures. Let us define $\psi_{x,r}$ as
\[
\psi_{x,r} (y) = \frac{y-x}{r} \: .
\]

\begin{dfn}[Approximate tangent space to a measure, see definition $2.79$ p.$92$ in \cite{ambrosio}]
Let $\mu$ be a positive Radon measure in $\R^n$. We say that $\mu$ has an \emph{approximate tangent space} $P$ with multiplicity $\theta \in \R_+$ at $x$ if $P \in \G$ is a $d$--plane such that
\[
\frac{1}{r^d} {\psi_{x,r}}_\# \mu \: \xrightharpoonup[]{\: \: \ast \: \:} \: \theta \cH^d_{| P} \text{ as } r \downarrow 0.
\]
That is,
\[
\frac{1}{r^d} \int \varphi \left(\frac{y-x}{r}\right) \, d \mu(y) \xrightarrow[r \downarrow 0]{} \theta \int_P \varphi(y) \, d \cH^d(y) \quad \forall \varphi \in \xC_c (\R^n) \: .
\]
\end{dfn}

\noindent In the sequel the approximate tangent plane to $M$ (resp. $\mu$) at $x$ is denoted by $T_x M$ (resp. $T_x \mu$). As we said, this provides a way to characterize rectifiability:

\begin{theo}[see theorem $2.83$ p.$94$ in \cite{ambrosio}]
Let $\mu$ be a positive Radon measure in $\R^n$.
\begin{enumerate}
\item If $\mu = \theta \cH^d_{| M}$ with $M$ countably $d$--rectifiable, then $\mu$ admits an approximate tangent plane with multiplicity $\theta (x)$ for $\cH^d$--almost any $x \in M$.
\item If there exists a Borel set $S$ such that $\mu(\R^n \setminus S)=0$ and if $\mu$ admits an approximate tangent plane with multiplicity $\theta (x) > 0$ for $\cH^d$--almost every $x \in S$ then $S$ is countably $d$--rectifiable and $\mu = \theta \cH^d_{| S}$.
\end{enumerate}
\end{theo}

There are other characterizations of rectifiability in terms of density (see for instance \cite{Mattila19951}). Let us point out an easy consequence of the existence of a tangent plane at a given point:

\begin{prop}\label{approximate_plane_consequence}
Let $\mu$ be a positive Radon measure in $\R^n$. Let $x \in \R^n$, $P \in \G$ and assume that $\mu$ has an approximate tangent space $T_x \mu$ with multiplicity $\theta(x)>0$ at $x$. Then for all $\beta >0$,
\[
\frac{1}{r^d} \mu \left\lbrace y \in B_r (x) \, | \, d(y-x,P) < \beta r  \right\rbrace \xrightarrow[r \to 0]{} \theta(x) \cH^d \left\lbrace y \in T_x \mu \cap B_1 (0)  \, | \, d(y,P) < \beta \right\rbrace \: .
\]
\end{prop}

\begin{proof}
Indeed, let $\psi_{x,r} : y \mapsto \frac{y-x}{r}$, then $\frac{1}{r^d} {\psi_{x,r}}_\# \mu$ weakly star converges to $\theta (x) \cH^d_{| _x \mu}$ so that for any Borel set $A$ such that $\cH^d_{| T_x \mu}(\partial A) = \cH^d (\partial A \cap T_x \mu) = 0$, we have
\begin{equation} \label{def_plan}
\frac{1}{r^d} {\psi_{x,r}}_\# \mu (A) = \frac{1}{r^d} \mu \left( \psi_{x,r}^{-1} (A) \right) \xrightarrow[r \to 0_+]{}  \theta (x) \cH^d \left( T_x \mu \cap A \right) \: .
\end{equation}
The conclusion follows applying (\ref{def_plan}) with $A = \left\lbrace y \in B_1(0) \, | \, d(y,P) < \beta \right\rbrace$ so that for any $0 < \beta < 1$, 
\[
\psi_{x,r}^{-1} (A) = \left\lbrace y \in B_r (x) \, | \, d(y-x,P) > \beta r  \right\rbrace \text{ and } \cH^d (A \cap P) = 0 \: .
\]
\end{proof}

\subsection{Some facts about varifolds}

We recall here a few facts about varifolds, (for more details, see for instance \cite{MR756417}). As we have already mentioned, the space of varifolds can be seen as a space of generalized surfaces. However, in this part we give examples showing that, not only rectifiable sets, but also objects like point clouds or volumetric approximations can be endowed with a varifold structure. Then we define the first variation of a varifold which is a generalized notion of mean curvature, and we recall the link between the boundedness of the first variation and the rectifiability of a varifold. We also introduce a family of volumetric discretizations endowed with a varifold structure. They will appear all along this paper in order to illustrate problems and strategies to solve them. We focus on this particular family of varifolds because they correspond to the volumetric approximations of sets that motivated us initially.

\subsubsection{Definition of varifolds}

We recall that $\G = \left\lbrace P \subset \R^n \, | \, P \text{ is a vector subspace of dimension } d \right\rbrace$. Let us begin with the notion of rectifiable $d$--varifold.

\begin{dfn}[Rectifiable $d$--varifold]
Given an open set $\Omega \subset \R^n$, let $M$ be a countably $d$--rectifiable set and $\theta$ be a non negative function with $\theta > 0$ $\cH^d$--almost everywhere in $M$. A rectifiable $d$--varifold $V= v(M,\theta)$ in $\Omega$ is a positive Radon measure on $\Omega \times G_{d,n}$ of the form $V= \theta \mathcal{H}^d_{| M} \otimes \delta_{T_x M}$ i.e.
\[
\int_{\Omega \times G_{d,n}} \varphi (x,T) \, dV(x,T) = \int_M \varphi (x, T_x M) \, \theta(x) \, d \mathcal{H}^d (x) \quad \forall \varphi \in \xC_c (\Omega \times G_{d,n} , \mathbb{R})
\] where $T_x M$ is the approximative tangent space at $x$ which exists $\mathcal{H}^d$--almost everywhere in $M$. The function $\theta$ is called the \emph{multiplicity} of the rectifiable varifold.
\end{dfn}
\begin{remk}
We are dealing with measures on $\Omega \times G_{d,n}$, but we did not mention the $\sigma$--algebra we consider. We can equip $G_{d,n}$ with the metric
\[
d(T,P) = \Vert \Pi_T - \Pi_P \Vert
\] with $\Pi_T \in M_n (\R)$ being the matrix of the orthogonal projection onto $T$ and $\Vert \cdot \Vert$ a norm on $M_n (\R)$. We consider measures on $\Omega \times G_{d,n}$ with respect to the Borel algebra on $\Omega \times G_{d,n}$.
\end{remk}
\noindent Let us turn to the general notion of varifold:
\begin{dfn}[Varifold]
Let $\Omega \subset \R^n$ be an open set. A $d$--varifold in $\Omega$ is a positive Radon measure on $\Omega \times G_{d,n}$.
\end{dfn}
\begin{remk}
As $\Omega \times G_{d,n}$ is locally compact, Riesz Theorem allows to identify Radon measures on $\Omega \times G_{d,n}$ and continuous linear forms on $\xC_c (\Omega \times G_{d,n})$ (we used this fact in the definition of rectifiable $d$--varifolds) and the convergence in the sense of varifolds is then the weak--$\ast$ convergence.
\end{remk}
\begin{dfn}[Convergence of varifolds]
A sequence of $d$--varifolds $(V_i)_i$ weakly--$\ast$ converges to a $d$--varifolds $V$ in $\Omega$ if, for all $\varphi \in \xC_c (\Omega \times G_{d,n})$,
\[
\int_{\Omega \times \G} \phi (x,P) \, dV_i (x,P) \xrightarrow[i \to \infty]{} \int_{\Omega \times \G} \phi(x,P) \, dV(x,P) \: .
\]
\end{dfn}
We now give some examples of varifolds:
\begin{xmpl}
\label{droite}
Consider a straight line $D \subset \mathbb{R}^3$, then the measure $v(D) = \mathcal{H}^1_{|D} \otimes \delta_D$ is the canonical $1$--varifold associated to $D$.
\end{xmpl}
\begin{xmpl} \label{ligne_polygonale}
Consider a polygonal curve $M \subset \mathbb{R}^2$ consisting of $8$ line segments $S_1, \ldots , S_8$ of directions $P_1 , \ldots , P_8 \in G_{1,2}$, then the measure $v(M) = \sum_{i=1}^8 \mathcal{H}^1_{|S_i} \otimes \delta_{P_i}$ is the canonical varifold associated to $M$. 
\end{xmpl}
\setcounter{subfigure}{0}
\begin{figure}[!htp]
\subfigure[Polygonal curve]{\includegraphics[width=0.45\textwidth]{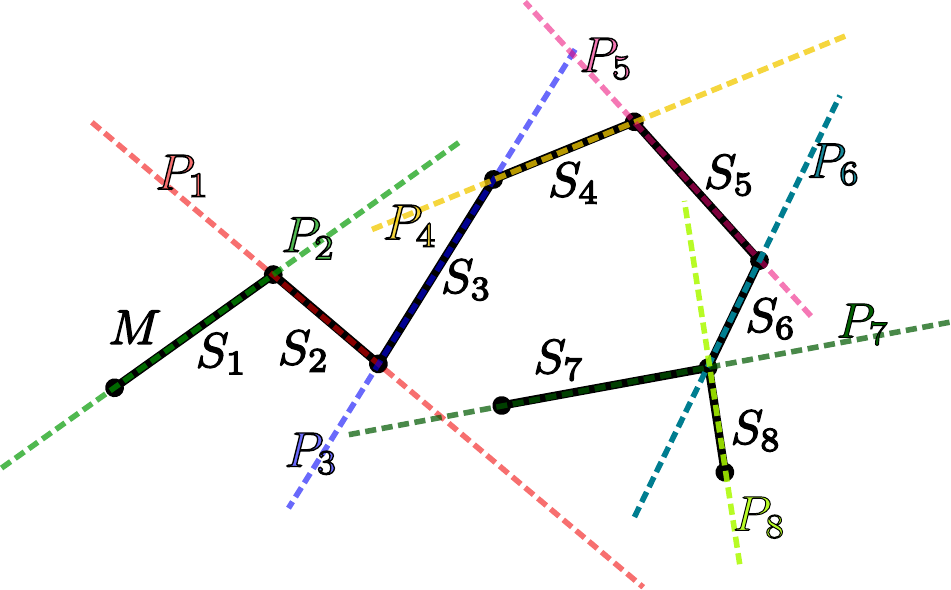}}
\subfigure[Point cloud]{\includegraphics[width=0.45\textwidth]{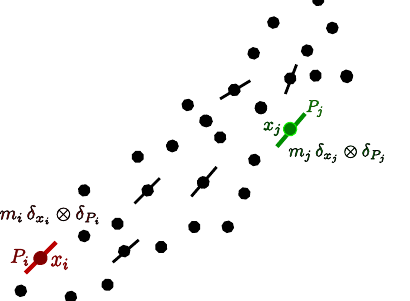}}
\end{figure}
\begin{xmpl}
Consider a $d$--submanifold $M \subset \R^n$. According to the definition of rectifiable $d$--varifolds, the canonical $d$--varifold associated to $M$ is $v(M) = \cH^d \otimes \delta_{T_x M}$ or $v(M,\theta) = \theta \cH^d \otimes \delta_{T_x M} $ adding some multiplicity $\theta : M \rightarrow \R_+$.
\end{xmpl}
\begin{xmpl}[Point cloud]
Consider a finite set of points $\{ x_j \}_{j = 1}^N \subset \R^n$ with additional information of masses $\{ m_j \}_{j = 1}^N  \subset \R_+$ and tangent planes $\{ P_j \}_{j = 1 \ldots N} \subset \G$ then the measure
\[
\sum_{j=1}^N m_j \delta_{x_j} \otimes \delta_{P_j} 
\]
defines a $d$--varifolds associated with the point cloud.
\end{xmpl}

\begin{dfn}[Mass]
If $V=v(M,\theta)$ is a $d$--rectifiable varifold, the measure $\theta \mathcal{H}^d_{| M}$ is called the \emph{mass} of $V$ and denoted by $\V$. For a general varifold $V$, the mass of $V$ is the positive Radon measure defined by $\Vert V \Vert (B) = V (\pi^{-1} (B))$ for every $B \subset \Omega$ Borel, with \[ \left\lbrace \begin{array}{lcll}
\pi : & \Omega \times G_{d,n} & \rightarrow & \Omega \\
& (x,S) & \mapsto & x
\end{array} \right. \]
\end{dfn}
\noindent For a curve, the mass is the length measure, for a surface, it is the area measure, for the previous point cloud, the mass is $\sum_j m_j \delta_{x_j}$. The mass loses the tangent information and keeps only the spatial part.

\subsubsection{First variation of a varifold}

The set of $d$--varifolds is endowed with a notion of generalized curvature called \emph{first variation}. Let us recall the divergence theorem on a submanifold:
\begin{theo}[Divergence theorem]
Let $\Omega \subset \R^n$ be an open set and let $M \subset \R^n$ be a $d$--dimensional $\xC^2$-- submanifold. Then, for all $X \in \xC_c^1 (\Omega , \R^n )$,
\[
\int_{ \Omega \cap M } {\rm{div}}_{T_x M} X(x) \, d\cH^d(x) = - \int_{\Omega \cap M} H(x) \cdot X(x) \, d\cH^d(x) \: ,
\]
where $H$ is the mean curvature vector.
\end{theo}
\noindent For $P \in G$ and $X =(X_1, \ldots , X_n) \in \xC_c^1 (\Omega , \R^n )$, the operator ${\rm{div}}_P$ is defined as
\[
{\rm{div}}_P (x) = \sum_{j=1}^n \langle \nabla^P X_j (x) , e_j \rangle = \sum_{j=1}^n \langle \Pi_P (\nabla X_j (x)) , e_j \rangle \text{ whith } (e_1, \ldots, e_n) \text{ canonical basis of } \R^n .
\]
This variational approach is actually a way to define mean curvature that can be extended to a larger class than $C^2$--manifolds: the class of varifolds with bounded first variation. We can now define the first variation of a varifold. 
\begin{dfn}[First variation of a varifold] The first variation of a $d$--varifold in $\Omega \subset \R^n$ is the linear functional
\[
\begin{array}{lcll}
\delta V : & {\xC}_c^1 (\Omega , \R^n )& \rightarrow & \R \\
& X & \mapsto & \int_{\Omega \times G_{d,n}} {\xdiv}_P X (x) \, d V (x,P)
\end{array}
\]
\end{dfn}
\noindent This linear functional is generally not continuous with respect to the $\xC_c^0$ topology. When it is true, we say that the varifold has \emph{locally bounded first variation}:
\begin{dfn}
We say that a $d$--varifold on $\Omega$ has locally bounded first variation when the linear form $\delta V$ is continuous that is to say, for every compact set $K \subset \Omega$ there is a constant $c_K$ such that for every $X \in \xC_c^1 (\Omega , \R^n)$ with $\supp X \subset K$,
\[
| \delta V (X) | \leq c_K \sup_K |X| \: .
\] 
\end{dfn}
Now, if a $d$--varifold $V$ has locally bounded first variation, the linear form $\delta V$ can be extended into a continuous linear form on $\xC_c (\Omega , \R^n )$ and then by Riesz Theorem, there exists a Radon measure on $\Omega$ (still denoted $\delta V$) such that
\[
\delta V (X) = \int_\Omega X \cdot \delta V \quad \text{for every } X \in \xC_c(\Omega,\mathbb{R}^n)
\]
Thanks to Radon-Nikodym Theorem, we can derive $\delta V$ with respect to $\| V \|$ and there exist a function $H \in \left( \xL^1_{loc}(\Omega, \| V \|) \right)^n$ and a measure $\delta V_s$ singular to $\| V \|$ such that
\[
\delta V = - H \|V\| + \delta V_s \: .
\]
The function $H$ is called the generalized mean curvature vector. Thanks to the divergence theorem, it properly extends the classical notion of mean curvature for a $\xC^2$ submanifold.

\subsubsection{Another example: a family of volumetric approximations endowed with a varifold structure} \label{section_varifolds_discrets}

Let us explain what we mean by volumetric approximation. Given a $d$--rectifiable set $M \subset \R^n$ (a curve, a surface...) and a mesh $\cK$, we can define for any cell $K \in \cK$, a mass $m_K$ (the length of the piece of curve in the cell, the area of the piece of surface in the cell) and a mean tangent plane $P_K$ as
\[
m_K = \cH^d (M \cap K) \text{ and } P_K \in \argmin_{S \in \G} \int_{M \cap K} \left| T_x M - S \right|^2 \, d \cH^d (x) \: ,
\]
and similarly, given a rectifiable $d$--varifold $V$, defining
\[
m_K = \V ( K) \text{ and } P_K \in \argmin_{S \in \G} \int_{K \times \G} \left| P - S \right|^2 \, dV(x,P) \: ,
\]
gives what we call a volumetric approximation of  $V$. We now introduce the family of varifolds of this form:

\noindent \begin{minipage}{0.55\textwidth}
\begin{xmpl}\label{diffuse_discrete_varifolds}
Consider a mesh $\cK$ and a family $\{m_K ,P_K  \}_{K \in \cK} \subset \R_+ \times \G$. We can associate the diffuse $d$--varifold:
\[
V = \sum_{K \textrm{cell}} \frac{m_K}{|K|} \cL^n_{| K} \otimes \delta_{P_K} \text{ with } |K| = \cL^n (K) \: .
\]
This $d$--varifold is not rectifiable since its support is $n$--rectifiable but not $d$--rectifiable. We will refer to the set of $d$--varifolds of this special form as \emph{discrete varifolds}.
\end{xmpl}
\end{minipage}
\begin{minipage}{0.45\textwidth} 
\includegraphics[width=0.90\textwidth]{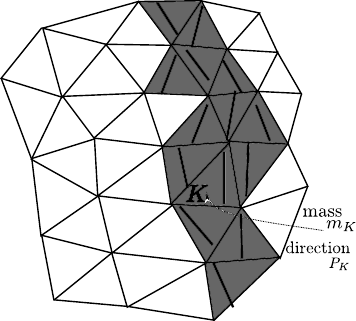}
\end{minipage}
Let us now compute the first variation of such a varifold:
\begin{prop} \label{computation_first_variation_discrete_varifold}
Let $\cK$ be a mesh of $\R^n$ and denote $\cE$ the set of faces of $\cK$. For $K_+, \, K_- \in \cK$, we denote by $\sigma = K_+ | K_- \in \cE$ the common face to $K_+$ and $K_-$, and $n_{K_+, \sigma}$ is then the outer-pointing normal to the face $\sigma$ (pointing outside $K_+$). Decompose the set of faces into $\cE = \cE_{int} \cup \cE_b \cup \cE$ where
\begin{itemize}
\item $\cE_{int}$ is the set of faces $\sigma = K_+ | K_-$ such that $m_{K_+}$, $m_{K_-} > 0$, called \emph{internal faces},
\item $\cE_0$ is the set of faces $\sigma = K_+ | K_-$ such that $m_{K_+}$, $m_{K_-} = 0$,
\item $\cE_b$ is the set of remaining faces $\sigma = K_+ | K_-$ such that $m_{K_+} > 0$ and $m_{K_-} = 0$ or conversely $m_{K_+} = 0$ and $m_{K_-} > 0$, called \emph{boundary faces}. In this case, $\sigma$ is denoted by $K_+ | \cdot$ with $m_{K_+} >0$.
\end{itemize}
For $\{m_K ,P_K  \}_{K \in \cK} \subset \R_+ \times \G$, let us define the $d$--varifold
\[
V_\cK = \sum_{K \in \cK} \frac{m_K}{|K|} \cL^n_{| K} \otimes \delta_{P_K} \: .
\]
Then,
\[
| \delta V_\cK |= \sum_{\substack{\sigma \in \cE_{int}, \\ \sigma = K_- | K_+}} \left| \left[ \frac{m_{K_+}}{|K_+|} \Pi_{P_{K_+}}  - \frac{m_{K_-}}{|K_-|} \Pi_{P_{K_-}} \right] (n_{K_+,\sigma} ) \right| \,  \cH^{n-1}_{| \sigma} + \sum_{\substack{\sigma \in \cE_b, \\ \sigma=K | \cdot}} \frac{m_{K}}{|K|} \left| \Pi_{P_{K}} n_{K,\sigma} \right| \, \cH^{n-1}_{\sigma} \: ,
\]
where $\Pi_P$ is the orthogonal projection onto the $d$-plane $P$.
\end{prop}

We stress that the terms \emph{internal faces} and \emph{boundary faces} do not refer to the structure of the mesh $\cK$ but to the structure of the support of $V_\cK$.

\begin{proof}
Let $\displaystyle V_\cK = \sum_{K \in \cK} \frac{m_K}{|K|} \cL^n_{| K} \otimes \delta_{P_K}$ be a discrete varidold associated with the mesh $\cK$ and let $X \in \xC_c^1 (\Omega , \R^n)$. Then,
\[
\delta V_\cK (X) = \int_{\Omega \times \G} \Div_S X(x) \, d V_\cK (x,S) = \sum_{K \in \cK} \frac{m_K}{|K|} \int_K \Div_{P_K} X(x) \, d \cL^n (x) \: .
\]
Let us compute this term. Fix $(\tau_1, \ldots , \tau_d)$ a basis of the tangent plane $P_\cK$ so that
\[
\int_K \Div_{P_\cK} X(x) \, d \cL^n (x) = \sum_{j=1}^d \int_K DX(x) \tau_j \cdot \tau_j \, d \cL^n (x) \: ,
\]
and $\displaystyle  DX(x) \tau_j \cdot \tau_j = \sum_{k=1}^n (\nabla X_k (x) \cdot \tau_j) \tau_j^k$ so that
\begin{align*}
\int_K \Div_{P_\cK} X(x) \, d \cL^n (x) & = \sum_{j=1}^d \sum_{k=1}^n \tau_j^k \int_K (\nabla X_k (x) \cdot \tau_j)  \, d \cL^n (x) = - \sum_{j=1}^d \sum_{k=1}^n \tau_j^k \int_{\partial K} X_k \tau_j \cdot n_{out} \, d \cH^d \\
& = - \int_{\partial K} \sum_{j=1}^d (\tau_j \cdot n_{out}) \sum_{k=1}^n  X_k  \tau_j^k  \, d \cH^d = - \int_{\partial K} \sum_{j=1}^d (\tau_j \cdot n_{out})( X \cdot \tau_j)  \, d \cH^d \\
& = - \int_{\partial K} X(x) \cdot (\Pi_{P_\cK} n_{out}) \, d \cH^d (x) \: ,
\end{align*}
where $\Pi_{P_\cK}$ is the orthogonal projection onto $P_\cK$ and $n_{out}$ is the outward-pointing normal. Consequently
\[
\left| \delta V_\cK (X) \right| = \left| \sum_{K \in \cK} \frac{m_K}{|K|} \int_{\partial K} X(x) \cdot (\Pi_{P_\cK} n_{out}) \, d \cH^d (x) \right| \leq \left\| X \right\|_{\infty} \sum_{K \in \cK} \frac{m_K}{|K|} \left| \Pi_{P_\cK} n_{out} \right| \cH^d(\partial K) \: .
\]
For a fixed mesh, the sum is locally finite and then, $V_\cK$ has locally bounded first variation. But what happens if the size of the mesh tends to $0$? In order to compute the total variation of $\delta V_\cK$ as a Radon measure, we just have to rewrite the sum as a sum on the faces $\cE$ of the mesh. This is more natural since $\delta V_\cK$ is concentrated on faces. Thus
\begin{align*}
\delta V_\cK & = - \sum_{\substack{\sigma \in \cE_{int}, \\ \sigma = K_- | K_+}} \left[ \frac{m_{K_+}}{|K_+|} \Pi_{P_{K_+}} n_{K_+,\sigma} + \frac{m_{K_-}}{|K_-|} \Pi_{P_{K_-}} n_{K_-,\sigma} \right]  \cH^{n-1}_{| \sigma} - \sum_{\substack{\sigma \in \cE_b, \\ \sigma=K | \cdot}} \frac{m_{K}}{|K|} \Pi_{P_{K}} n_{K,\sigma} \, \cH^{n-1}_{| \sigma} \\
& = - \sum_{\substack{\sigma \in \cE_{int}, \\ \sigma = K_- | K_+}} \left[ \frac{m_{K_+}}{|K_+|} \Pi_{P_{K_+}}  - \frac{m_{K_-}}{|K_-|} \Pi_{P_{K_-}} \right] \cdot (n_{K_+,\sigma} ) \,  \cH^{n-1}_{| \sigma} - \sum_{\substack{\sigma \in \cE_b, \\ \sigma=K | \cdot}} \frac{m_{K}}{|K|} \Pi_{P_{K}} n_{K,\sigma} \, \cH^{n-1}_{| \sigma}\: .
\end{align*}
Therefore,
\[
| \delta V_\cK | = \sum_{\substack{\sigma \in \cE_{int}, \\ \sigma = K_- | K_+}} \left| \left[ \frac{m_{K_+}}{|K_+|} \Pi_{P_{K_+}}  - \frac{m_{K_-}}{|K_-|} \Pi_{P_{K_-}} \right] \cdot (n_{K_+,\sigma} ) \right| \,  \cH^{n-1}_{| \sigma} + \sum_{\substack{\sigma \in \cE_b, \\ \sigma=K | \cdot}} \frac{m_{K}}{|K|} \left| \Pi_{P_{K}} n_{K,\sigma} \right| \, \cH^{n-1}_{| \sigma} \: .
\]
\end{proof}

\begin{xmpl} \label{first_variation_explosion}
Let us estimate this first variation in a simple case. Let us assume that the mesh is a regular cartesian grid of $\Omega = ]0,1[^2 \subset \R^2$ of size $h_\cK$ so that for all $K \in \cK$ and $\sigma \in \cE$,
\[
|K| = h_\cK^2 \text{ and } \cH^1 (\sigma) = h_\cK \: .
\]
Consider the vector line $D$ of direction given by the unit vector  $\frac{1}{\sqrt{2}}(1,1)$. Let $V = \cH^1_{| D} \otimes \delta_D$ be the canonical $1$--varifold associated to $D$ and $V_\cK$ the volumetric approximation of $V$ in the mesh $\cK$, then
\begin{align*}
| \delta V_\cK |(\Omega) & = \sum_{\substack{\sigma \in \cE_{int}, \\ \sigma = K_- | K_+}} \left| \left[ \frac{m_{K_+}}{|K_+|} \Pi_{P_{K_+}}  - \frac{m_{K_-}}{|K_-|} \Pi_{P_{K_-}} \right] \cdot (n_{K_+,\sigma} ) \right| \,  \cH^{1}(\sigma) + \sum_{\substack{\sigma \in \cE_b, \\ \sigma=K | \cdot}} \frac{m_{K}}{|K|} \left| \Pi_{P_{K}} n_{K,\sigma} \right| \, \cH^{1}(\sigma) \\
& =  \frac{1}{h_\cK} \sum_{\substack{\sigma \in \cE_{int}, \\ \sigma = K_- | K_+}}  \left|  m_{K_+} - m_{K_-} \right| \left| \Pi_D n_{K_+,\sigma} \right| + \frac{1}{h_\cK}\sum_{\substack{\sigma \in \cE_b, \\ \sigma=K | \cdot}} m_{K} \left| \Pi_D n_{K,\sigma} \right| \: .
\end{align*}
And $\left| \Pi_D n_{K,\sigma} \right| = \frac{\sqrt{2}}{2}$ (for any $K$, $\sigma$) so that
\[
| \delta V_\cK |(\Omega) =  \frac{\sqrt{2}}{2 h_\cK} \sum_{\substack{\sigma \in \cE_{int}, \\ \sigma = K_- | K_+}}  \left|  m_{K_+} - m_{K_-} \right|  + \frac{\sqrt{2}}{2 h_\cK} \underbrace{\sum_{\substack{\sigma \in \cE_b, \\ \sigma=K | \cdot}} m_{K} }_{= \V (\Omega)} \: .
\]
So that if we now consider successive volumetric approximations of $V_{\cK_i}$ associated with successive meshes $\cK_i$ whose size $h_{\cK_i}$ tends to $0$ when $i$ tends to $\infty$,
\[
| \delta V_{\cK_i} |(\Omega) =  \frac{\sqrt{2}}{2 h_{\cK_i}} \left( \sum_{\substack{\sigma \in \cE_{int}, \\ \sigma = K_- | K_+}}  \left|  m_{K_+} - m_{K_-} \right|  +  \V (\Omega) \right) \geq \frac{\sqrt{2}}{2 h_{\cK_i}} \V (\Omega) \xrightarrow[i \to \infty]{} +\infty \: .
\]
\end{xmpl}

\begin{minipage}{0.55\textwidth}
\includegraphics[width=0.95\textwidth]{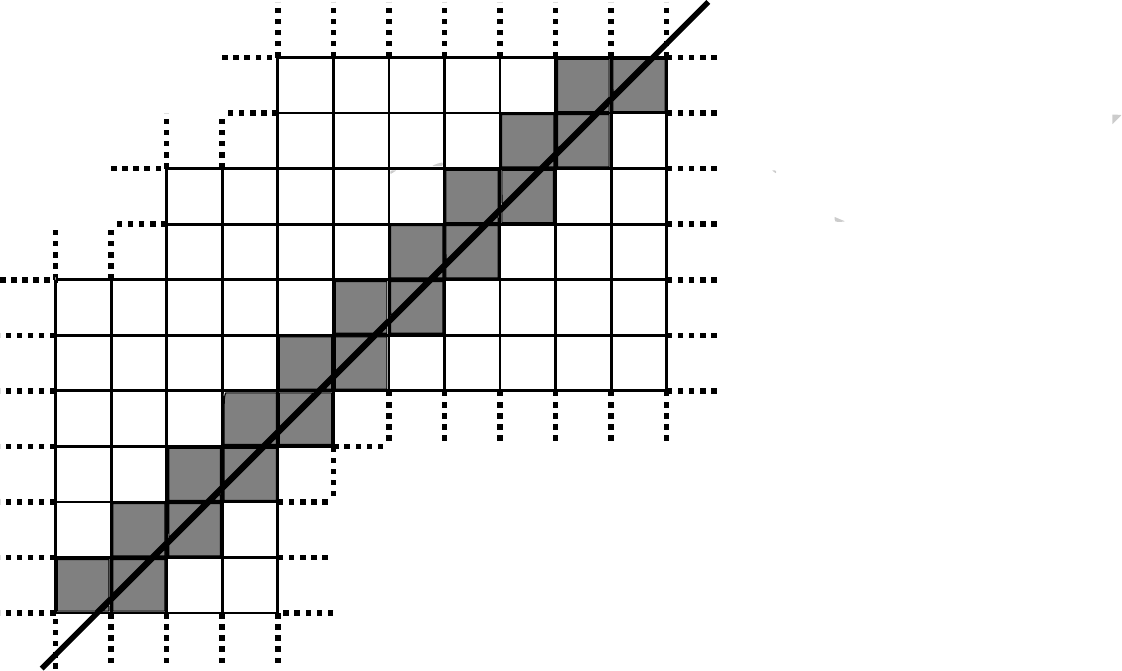}
\end{minipage}
\begin{minipage}{0.40\textwidth}
\includegraphics[width=0.95\textwidth]{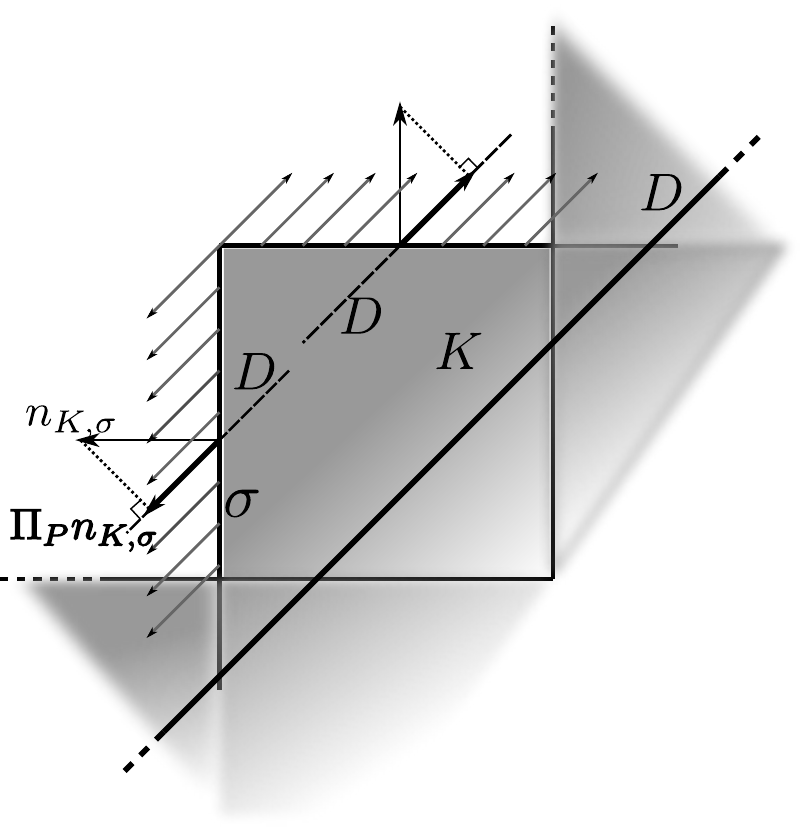}
\end{minipage}

More generally, the problem is that the tangential direction $P_{K}$ and the direction of the face $\sigma$ have no reason to be correlated so that the term $\left| \Pi_{P_{K}} n_{K,\sigma} \right|$ can be large (close to $1$) and thus, if the mesh is not adapted to the tangential directions $| \delta V_{\cK_i} |(\Omega) $ may explode when the size of the mesh $h_{\cK_i}$ tends to $0$. Of course, we are not saying that $| \delta V_{\cK_i} |(\Omega) $ always explodes when refining the mesh, but that it may happen and it is not something easy to control except by adapting the mesh to the tangential directions $P_\cK$ in the boundary cells. This is clearly a problem showing that the classical notion of first variation is not well adapted to this kind of volumetric discretization.

\subsubsection{Control of the first variation and rectifiability}

We will end these generalities about varifolds by linking the control of the first variation (generalized mean curvature) to the regularity of the varifolds. Let us begin with some property of the so called \emph{height excess} proved by Brakke in~\cite{MR485012} ($5.7$ p. $153$).

\begin{theo}[Height excess decay] \label{heightexcess}
Let $V = v(M,\theta) = \theta \cH^d_{| M} \otimes \delta_{T_x M}$ be a rectifiable $d$--varifold in some open set $\Omega \subset \R^n$. Assume that $V$ is integral (that is $\theta (x) \in \N$ for $\V$--almost every $x$) and assume that $V$ has locally bounded first variation. Then for $V$--almost every $(x,P) \in \Omega \times \G$,
\[
\heightex (x , P ,V , r) := \frac{1}{r^d} \int_{B_r(x)} \left( \frac{d(y-x,P)}{r} \right)^2 \, d \V (y) = o_x (r) \: .
\]
\end{theo}

\begin{remk} Let us notice that
\[
E_\alpha (x,P,V) = \int_{r=\alpha}^1  \heightex(x , P , V, r) \, \frac{dr}{r} \: .
\]
That is why we called these quantities \emph{averaged height excess}.
\end{remk}

We now state a compactness result linking the rectifiability to the control of the first variation. It is exactly the kind of result we are interested in, with the exception that, in our setting, the approximating varifolds are generally not rectifiable and, moreover, the following control on the first variation is not satisfied.

\begin{theo}[Allard Compactness Theorem, see $42.7$ in \cite{MR756417}]
Let $(V_i)_i = (v(M_i,\theta_i))_i$ be a sequence of $d$--rectifiable varifolds with locally bounded first variation in an open set $\Omega \subset \R^n$ and such that $\theta_i \geq 1 \: \: \| V_i \|$--almost everywhere. If
\[
\sup_i \left\lbrace \| V_i (W) \| + | \delta V_i | (W) \right\rbrace \leq c(W) < +\infty
\] for every open set $W \subset \subset \Omega$, then there exists a subsequence $(V_{i_n})_n$ weakly--$\ast$ converging to a rectifiable $d$--varifold $V$, with locally bounded first variation in $\Omega$, such that $\theta \geq 1$, and moreover
\[
| \delta V | (W) \leq \liminf_{n\to\infty} | \delta V_{i_n} | (W) \quad \forall \: W \subset \subset \Omega \: .
\]
If for all $i$, $V_i$ is an integral varifold then $V$ is integral too.
\end{theo}

The problem is that even if the limit $d$--varifold is rectifiable and has bounded first variation, it is not necessarily the case of an approximating sequence of varifolds. For instance, a point cloud varifold does not have bounded first variation. As for discrete $d$--varifolds of Example \ref{diffuse_discrete_varifolds}, we have computed the first variation and seen that it is bounded for a fixed mesh, however, when the size of the mesh tends to zero, the total variation of the first variation is no longer bounded (in general) because of some boundary terms. We need some other way to ensure rectifiability. That is why we are looking for something more volumetric than the first variation, as defined in the introduction, in order to enforce rectifiability:
\[
E_\alpha (x,P,V) = \int_{r=\alpha}^1 \frac{1}{r^d} \int_{y \in B_r(x) \cap \Omega} \left( \frac{d(y-x,P)}{r} \right)^2 \, d \V(y) \, \frac{dr}{r} \: .
\]
We now have two questions we want to answer:
\begin{enumerate}
\item Assume that $(V_i)_i$ is a sequence of $d$--varifolds weakly--$\ast$ converging to some $d$--varifold $V$ with the following control
\begin{equation}
\sup_i \int_{\Omega \times \G} E_{\alpha_i}(x,P,V_i) \, dV_i (x,P) < +\infty \: , \label{control}
\end{equation}
can we conclude that $V$ is rectifiable ?
\item Is this condition better adapted to the case of (non-rectifiable) volumetric approximating varifolds (i.e. sequences of discrete varifolds as defined in Example~\ref{diffuse_dicrete_varifolds} ? We will prove that as soon as $V_i$ weakly--$\ast$ converges to $V$, there exists a subsequence satisfying the control~(\ref{control}).
\end{enumerate}
We begin with studying the static case.

\section{Static quantitative conditions of rectifiability for varifolds} \label{static_section}

In this section, we begin with studying the averaged height excess $E_0 (x,P,V)$ with respect to $P\in \G$ (for a fixed $d$--varifold and a fixed $x \in \Omega$). We show that if $V$ has bounded first variation then the approximate tangent plane at $x$ is the only plane for which $E_0$ can be finite. Then we state and prove quantitative conditions of rectifiability for varifolds in the static case. Let us recall how we defined $E_0 (x,P,V)$ in Theorem~\ref{static_thm}.
\begin{dfn}[Averaged height excess]
Let $V$ be a $d$--varifold in $\Omega \subset \R^n$ open subset. Then we define
\[
E_0 (x,P,V) = \int_{r=0}^1 \frac{1}{r^d} \int_{y \in B_r(x) \cap \Omega} \left( \frac{d(y-x,P)}{r} \right)^2 \, d \V(y) \, \frac{dr}{r} \: .
\]
\end{dfn}
\noindent We first study the averaged height excess $E_0 (x,P,V)$ with respect to $P \in \G$ for a fixed rectifable $d$--varifold.

\subsection{The averaged height excess energy $E_0 (x,P,V)$}

Notice that if $\V = \cH^d_{| M}$ then for every $d$--vector plane $P \in \G$,
\begin{align*}
\int_{r=0}^1 \beta_2 (x,r,M)^2 \, \frac{dr}{r} & = \int_{r=0}^1 \inf_{S \in \left\lbrace \text{affine } d-\text{plane} \right\rbrace} \left( \frac{1}{r^d} \int_{y \in B_r(x) \cap M} \left( \frac{d(y,S)}{r} \right)^2 \, d \cH^d(y) \right) \, \frac{dr}{r} \\
& \leq \int_{r=0}^1 \frac{1}{r^d} \int_{y \in B_r(x) \cap M} \left( \frac{d(y-x,P)}{r} \right)^2 \, d \cH^d(y) \, \frac{dr}{r} = E_0 (x,P,V) \: .
\end{align*}
Thus, assume that for $\cH^d$--almost every $x \in M$, $\theta_\ast^d (x, M) > 0$ holds and that there exists some $P_x \in \G$ such that $E_0 (x,P_x,\cH^d_{| M}) < + \infty$. Then thanks to Pajot's Theorem \ref{Pajot}, $M$ is $d$--rectifiable. As we will see, the point is that for any $x \in M$ where the tangent plane $T_x M$ exists, then $P_x = T_x M$ is the best candidate, among all $d$--planes $P$, to satisfy $E_0 (x,P_x,\cH^d_{| M}) < + \infty$. Consequently, in order to test the rectifiability of a $d$--varifold $V$, it is natural to study $E_0 (x,P,V)$ for $(x,P)$ in $\supp V$ (which is more restrictive than for any $(x,P) \in \supp \V \times \G$). More concretely, we will study$ \displaystyle \int_{\Omega \times \G} E_0 (x,P,V) \, dV(x,P)$ rather than $\displaystyle \int_{\Omega} \inf_{P \in G} E_0 (x,P,V) \, d\V(x)$.

In this whole part, we fix a rectifiable $d$--varifold in some open set $\Omega \subset \R^n$ and we study the behaviour of $E_0 (x,P,V)$ with respect to $P \in \G$. We are going to show that for a rectifiable $d$--varifold, this energy is critical: under some assumptions, it is finite if and only if $P$ is the approximate tangent plane. More precisely:

\begin{prop} \label{min_of_continuous_energy}
Let $V = v(M,\theta)$ be a rectifiable $d$--varifold in an open set $\Omega \subset \R^n$. Then,
\begin{enumerate}
\item Let $x \in M$ such that the approximate tangent plane $T_x M$ to $M$ at $x$ exists and $\theta(x) >0$ (thus for $\V$--almost every $x$) then for all $P \in \G$ such that $P \neq T_x M$, 
\[
E_0 (x,P,V) = + \infty \: .
\]
\item If in addition $V$ is integral ($\theta \in \N$ $\V$--almost everywhere) and has bounded first variation then for $\V$--almost every $x$,
\[
E_0 (x,T_x M,V) < +\infty \: .
\]
\end{enumerate}
\end{prop}

\begin{proof}
We begin with the first assertion. Let $x \in M$ such that the approximate tangent plane  $T_x M$ to $M$ at $x$ exists. Let $P \in \G$ such that $P \neq T_x M$. Thanks to Prop.~\ref{approximate_plane_consequence}, for all $\beta >0$ we have
\[
\frac{1}{r^d} \V \left\lbrace y \in B_r (x) \, | \, d(y-x,P) < \beta r  \right\rbrace \xrightarrow[r \to 0_+]{} \theta(x) \cH^d \left( T_x M \cap \{ y \in B_1(0) \, | \, d(y,P) < \beta \} \right) \: .
\]
Now for all $\beta > 0$,
\begin{align*}
E_0 (x, P ,V) & = \int_{r=0}^1 \frac{dr}{r^{d+1}} \int_{B_r (x)} \left\lbrace \frac{d(y-x,P)}{r} \right\rbrace^2 \, d \V (y) \\
& \geq \int_{r=0}^1 \frac{dr}{r} \, \frac{1}{r^d} \int_{\left\lbrace y \in B_r (x) \, | \, d(y-x,P) \geq \beta r \right\rbrace} \beta^2 \, d \V (y) \\
& = \beta^2 \int_{r=0}^1 \frac{dr}{r} \, \frac{1}{r^d} \V \left\lbrace y \in B_r (x) \, | \, d(y-x,P) \geq \beta r \right\rbrace  \: .
\end{align*}
Let us estimate
\[
\frac{1}{r^d} \V \left\lbrace y \in B_r (x) \, | \, d(y-x,P) \geq \beta r \right\rbrace = \underbrace{ \frac{1}{r^d} \V (B_r (x)) }_{\xrightarrow[r\to 0]{} \theta (x) \omega_d } - \underbrace{ \frac{1}{r^d} \V \left\lbrace y \in B_r (x) \, | \, d(y-x,P) < \beta r \right\rbrace  }_{\xrightarrow[r \to 0]{} \theta (x) \cH^d (T_x M \cap \left\lbrace y \in B_1 (0) \, | \, d(y,P) < \beta \right\rbrace)  } \: .
\]
As $P \neq T_x M$, there exists some constant $c_P$ depending on $P$ and $T_x M$ such that
\[
\cH^d (T_x M \cap \left\lbrace y \in B_1 (0) \, | \, d(y,P) < \beta \right\rbrace) \leq c_P \beta \: .
\]
\begin{figure}
\centering
\includegraphics[width=0.60\textwidth]{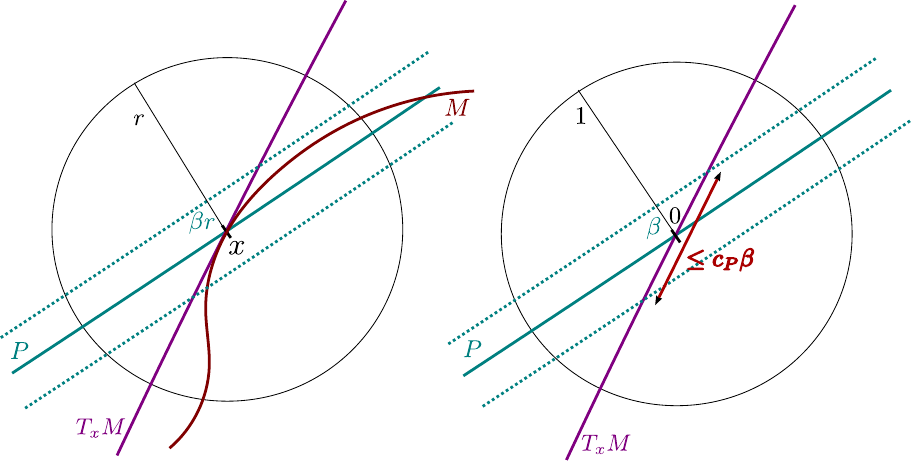}
\end{figure}
Consequently,
\begin{align*}
\lim_{r \to 0} \frac{1}{r^d} \V \left\lbrace y \in B_r (x) \, | \, d(y-x,P) \geq \beta r \right\rbrace & = \theta (x) \left( \omega_d - \cH^d (T_x M \cap \left\lbrace y \in B_1 (0) \, | \, d(y,P) < \beta \right\rbrace) \right) \\
& \geq \theta (x) ( \omega_d - c_P \beta) \\
& \geq \theta (x) \frac{\omega_d}{2} \text{ for } \beta \text{ small enough. }
\end{align*}
Eventually there exist $\beta>0$ and $r_0>0$ such that for all $r \leq r_0$ 
\[
\frac{1}{r^d} \V \left\lbrace y \in B_r (x) \, | \, d(y-x,P) \geq \beta r \right\rbrace \geq \theta(x)\frac{\omega_d}{4} \: ,
\]
and thus
\[
E_0 (x,P,V ) \geq \theta(x)\frac{\omega_d}{4} \beta^2 \int_{r=0}^{r_0} \frac{dr}{r} = +\infty \: .
\]
The second assertion is a direct consequence of Brakke's estimate (see Proposition \ref{heightexcess}) for the height excess of an integral $d$--varifold with bounded first variation:
\[
E_0 (x,T_x M,V)  = \int_{r=0}^1  \underbrace{\frac{1}{r} \heightex(x , P , V, r) }_{= o_x(1)} \, dr  < + \infty \: .
\]
\end{proof}

\subsection{The static theorem}

We begin with some lemmas before proving the static theorem (Theorem. \ref{static_thm}). This first proposition recalls that the first assumption of the static theorem (Ahlfors regularity) implies that $\V$ is equivalent to $\cH^d_{| \supp \V} $.

\begin{prop} \label{density_2}
Let $\Omega \subset \R^n$ be an open set and $\mu$ be a positive Radon measure in $\Omega$.
\begin{enumerate}[(i)]
\item Let $\beta_1, \, \beta_2 : \Omega \rightarrow \R_+$ continuous and such that for all $x \in \Omega$, $\beta_1 (x) < \beta_2(x)$, and let $C>0$. Then the sets $A = \left\lbrace x \in \Omega \: | \: \forall r \in \left( \beta_1(x) ,\beta_2(x)\right) , \: \mu (B_r(x)) \geq C r^d \right\rbrace$\\ and $B = \left\lbrace x \in \Omega \: | \: \forall r \in \left( \beta_1(x) ,\beta_2(x)\right), \: \mu (B_r(x)) \leq C r^d \right\rbrace$ are closed.
\item If there exist $C_1$, $C_2>0$ such that $C_1 \omega_d r^d \leq \mu (B_r(x))  \leq C_2 \omega_d r^d$ for $\mu$--almost all $x \in \Omega$ and for all $0 < r < d(x,\Omega^c)$, then
\[
C_1 \cH^d_{| E} \leq \mu \leq 2^d C_2 \cH^d_{| E}  \quad \text{with } E = \supp \mu \: .
\]
\end{enumerate}
\end{prop}

\begin{proof}
\begin{enumerate}[(i)]
\item Let us prove that $A = \left\lbrace x \in \Omega \: | \: \forall r \in \left( \beta_1(x) ,\beta_2(x) \right), \: \mu (B_r(x)) \geq C r^d \right\rbrace$ is closed. Let $(x_k)_k \subset A$ such that $x_k \xrightarrow[k \infty]{} x \in \Omega$ and let $r > 0$ such that $\beta_1 (x) < r < \beta_2 (x)$. For $k$ great enough, $\beta_1 (x_k) < r < \beta_2(x_k)$ so that $C r^d \leq \mu (B_r(x_k))$.
If $\mu (\partial B_r(x)) = 0$ then $\mu (B_r(x_k)) \xrightarrow[k \to +\infty]{} \mu (B_r(x))$ and then $C r^d \leq \mu (B_r(x))$ for almost every $r \in (\beta_1(x),\beta_2(x))$. But this is enough to obtain the property for all $r \in (\beta_1(x),\beta_2(x))$. Indeed, if $\mu (\partial B_r(x)) > 0$ then take $r^-_k < r $ such that for all $k$,
\[
\mu (\partial B_{r^{-}_k}(x)) = 0 \text{ and } r^{-}_k \xrightarrow[k \to +\infty]{} r \: ,
\]
and thus
\[
\mu (B_r(x)) \geq \mu (B_{r^-_k}(x)) \geq C {r^-_k}^d \xrightarrow[k \to +\infty]{} C r^d \: .
\]
Eventually $x \in A$ and $A$ is closed. We can prove that $B$ is closed similarly.
\item As the set
\[
E_1 = \left\lbrace x \in \Omega \: | \: \forall 0 < r < d(x,\Omega^c), \: \mu (B_r(x)) \geq C_1 \omega_d r^d \right\rbrace
\]
is closed (thanks to $(i)$) and of full $\mu$--measure, then $E = \supp \mu \subset E_1$. Therefore, for every $x \in E$,
\[
\theta_\ast^d (\mu,x) = \liminf_{r \to 0_+} \frac{\mu(B_r(x))}{\omega_d r^d} \geq C_1 \: . 
\]
So that (see Theorem $2.56$ p.$78$ in \cite{ambrosio}) $\mu \geq C_1 \cH^d_{| E}$.
\item For the same reason,
\[
E = \supp \mu \subset E_2 = \left\lbrace x \in \Omega \: | \: \forall 0 < r < d(x,\Omega^c), \: \mu (B_r(x)) \leq C_2 \omega_d r^d \right\rbrace \: .
\]
Therefore, for every $x \in E$,
\[
\theta^{\ast \, d} (\mu,x) = \limsup_{r \to 0_+} \frac{\mu(B_r(x))}{\omega_d r^d} \leq C_2 \: . 
\]
So that (again by Theorem $2.56$ p.$78$ in \cite{ambrosio}) $\mu \leq 2^d C_2 \cH^d_{| E}$.
\end{enumerate}
\end{proof}

\noindent The following lemma states that under some density assumption, the quantity $\min_{P \in \G} E_0 (x,P,V)$ controls the quantity linked to Jones' $\beta$ numbers.

\begin{lemma} \label{lemma_rectif_1}
Let $\Omega \subset \R^n$ be an open set and let $V$ be a $d$--varifold in $\Omega$. Assume that there is some constant $C>0$ and a Borel set $E \subset \Omega$ such that $\cH^d_{| E} \leq C \V$ then for all $x \in \Omega$,
\begin{equation}
\int_0^1 \beta_2(x,r,E)^2 \frac{dr}{r} \leq C \min_{P \in \G} E_0 (x,P,V) \: . \label{link_jones_E_0}
\end{equation}
\end{lemma}

\begin{proof}
First notice that $\G \subset \left\lbrace \textrm{affine $d$--plane} \right\rbrace$, therefore
\begin{align*}
\int_{r=0}^1 \beta_2(x,r,E)^2 \frac{dr}{r^{d+1}} & = \int_{r=0}^1 \inf_{P \in \left\lbrace \textrm{affine $d$--plane} \right\rbrace } \left(  \int_{E \cap B_r(x)}  \left( \frac{d(y,P)}{r} \right)^2 \, d \cH^d(y) \right) \frac{dr}{r^{d+1}} \\
& \leq \inf_{P \in \left\lbrace \textrm{affine $d$--plane} \right\rbrace }  \int_{r=0}^1 \left(  \int_{E \cap B_r(x)}  \left( \frac{d(y,P)}{r} \right)^2 \, d \cH^d(y) \right) \frac{dr}{r^{d+1}} \\
& \leq \min_{P \in \G} \int_{r=0}^1 \left(  \int_{E \cap B_r(x)}  \left( \frac{d(y-x,P)}{r} \right)^2 \, d \cH^d(y) \right) \frac{dr}{r^{d+1}} \: .
\end{align*}
Then, the assumption $ \cH^d_{| E} \leq C \V$ implies that for any positive function $u$, $\displaystyle \int_E u \, d \cH^d \leq C \int_\Omega u \, d\V$ so that 
\[
\min_{P \in \G} \int_{r=0}^1 \left(  \int_{B_r(x)}  \left( \frac{d(y-x,P)}{r} \right)^2 \, d \cH^d_{| E}(y) \right) \frac{dr}{r^{d+1}} \leq C \min_{P \in \G} E_0 (x,P,V) \: ,
\]
which proves \ref{link_jones_E_0}.
\end{proof}

\noindent We now state a lemma that will enable us to localise the property of rectifiability.

\begin{lemma} \label{lemma_rectif_2}
Let $\Omega \subset \R^n$ be an open set and $\mu$ be a positive Radon measure in $\Omega$. Then there exists a countable family of open sets $(\omega_n)_n$ such that for all $n$, $\omega_n \subset\subset \omega_{n+1} \subset\subset \Omega$, $\mu (\partial \omega_n) = 0$ and $\Omega = \cup_n \omega_n$.
\end{lemma}

\begin{proof}
For all $t>0$, let us consider the family of open sets
\[
\omega_t = B_t (0) \cap \left\lbrace x \in \Omega \: | \: d(y,\Omega^c) > 1/t \right\rbrace \: .
\]
The family $(\omega_t)_t$ is increasing so that $\mu (\omega_t)$ is increasing and has at most a countable number of jumps. Then for almost every $t$, $\mu (\omega_t) =0$ and it is easy to conclude.
\end{proof}


\noindent The last step before proving Theorem~\ref{static_thm} is to link the rectifiability of the mass $\V$ and the rectifiability of the whole varifold. The key point is the coherence between the tangential part of the varifold and the approximate tangent plane to the spatial part $\V$.

\begin{lemma} \label{lemma_rectif_3}
If $V$ is a $d$--varifold in $\Omega \subset \R^n$ such that
\begin{enumerate}[$-$]
\item $\V$ is $d$--rectifiable,
\item $V \left( \left\lbrace (x,P) \in \Omega \times \G \: | \: E_0 (x,P,V ) = +\infty \right\rbrace \right) = 0$,
\end{enumerate}
then $V$ is a rectifiable $d$--varifold.
\end{lemma}

\begin{proof}
The mass $\V$ is $d$--rectifiable so that $\V=\theta \cH^d_{| M}$ for some $d$-rectifiable set $M$. We have to show that $V = \V \otimes \delta_{T_x M}$. Applying a disintegration theorem (\cite{ambrosio} $2.28$ p. $57$), there exist finite Radon measures $\nu_x$ in $\G$ such that for $\V$--almost every $x \in \Omega$, $\nu_x (\G) = 1$ and $V = \V \otimes \nu_x$. We want to prove that for $\V$--almost every $x$, $\nu_x = \delta_{T_x M}$ or equivalently,
\[
\nu_x(\left\lbrace P \in \G \: | \: P \neq T_x M \right\rbrace) = 0 \: .
\]
For a $d$--rectifiable measure $\V = \theta \cH^d_{| M}$, we have shown in Proposition \ref{min_of_continuous_energy} that for $\V$--almost every $x \in \Omega$,
\[
P \neq T_x M \Longrightarrow E_0 (x,P,V) = +\infty \: ,
\]
thus $\left\lbrace (x,P) \in \Omega \times \G \: | \: P \neq T_x M \right\rbrace \subset A_0 \times \G \cup \left\lbrace (x,P) \in \Omega \times \G \: | \: E_0 (x,P,V)=+\infty \right\rbrace$ with $\V(A_0) = 0$. Therefore $V (\left\lbrace (x,P) \in \Omega \times \G \: | \: P \neq T_x M \right\rbrace)=0$. Thus
\begin{align*}
V (\left\lbrace (x,P) \in \Omega \times \G \: | \: P \neq T_x M \right\rbrace) & = \int_{\Omega \times \G} \mathds{1}_{\{ P \neq T_x M \} } (x,P) \, dV(x,P) \\
& = \int_\Omega \left( \int_{\G} \mathds{1}_{ \{ P \neq T_x M \} } (x,P) \, d \nu_x(P) \right) d\V(x) \\
& = \int_\Omega \nu_x (\left\lbrace P \in \G \: | \: P \neq T_x M \right\rbrace) \, d\V(x)
\end{align*}
which means that for $\V$--almost every $x \in \Omega$, $\nu_x(\left\lbrace P \in \G \: | \: P \neq T_x M \right\rbrace) = 0$ thus for $\V$--almost every $x \in \Omega$, $\nu_x = \delta_{T_x M}$ and $V = \V \otimes \delta_{T_x M}$ is a $d$--rectifiable varifold.
\end{proof}
Let us now prove the static theorem:
\setcounter{theo}{2}
\begin{theo} \label{static_theorem}
Let $\Omega \subset \R^n$ be an open set and let $V$ be a $d$--varifold in $\Omega$ of finite mass $\V (\Omega) < +\infty$. Assume that:
\begin{enumerate}[(i)]
\item there exist $0 < C_1 < C_2$ such that for $\V$--almost every $x \in \Omega$ and for all $0 < r < d(x,\Omega^c)$ such that $B_r(x) \subset \Omega$,
\[
C_1 \omega_d r^d \leq \| V \| (B_r(x)) \leq C_2 \omega_d r^d \: ,
\]
\item $V\left( \left\lbrace (x,P) \in \Omega \times \G \, | \, E_0(x,P,V) = + \infty \right\rbrace \right) = 0$.
\end{enumerate}
Then $V$ is a rectifiable $d$--varifold.
\end{theo}

\begin{remk}
If in particular $\displaystyle \int_{\Omega \times \G} E_0 (x,P,V) \, d V(x,P) < +\infty $ then the assumption $(ii)$ is satisfied.
\end{remk}

\begin{proof}
Now we just have to gather the previous arguments and apply Pajot's Theorem (Theoremm.~\ref{Pajot}).
\begin{enumerate}[$-$]
\item Step $1$: First hypothesis implies (thanks to Proposition~\ref{density_2}) that, setting $C_3 = 2^d C_2 > 0$ and $E = \supp \V$, we have
\[
C_1 \cH^d_{| E} \leq \V \leq C_3 \cH^d_{| E} \: .
\]
Hence $C_1 \cH^d (E ) \leq \V(\Omega) < + \infty$. Moreover, as $\V$ and $\cH^d_{| E}$ are Radon measures and $\V$ is absolutely continuous with respect to $\cH^d_{ E}$, then by Radon-Nikodym Theorem there exists some function $\theta \in \xL^1(\cH^d_{| E})$ such that
\[
\V = \theta \cH^d_{| E} \quad \text{with} \quad \theta(x) = \frac{d \V}{d \cH^d_{| E}}(x) = \lim_{r \to 0_+} \frac{\V (B_r(x))}{\cH^d (E \cap B_r(x))} \geq C_1 > 0 \text{ for } \cH^d \text{ a.e. } x \in E \: .
\] 
\item Step $2$: Thus we can now apply Lemma~\ref{lemma_rectif_1} so that for any $x \in \Omega$,
\[
\int_0^1 \beta_2(x,r,E)^2 \frac{dr}{r} \leq C_3 \min_{P \in \G} E_0 (x,P,V) \: ,
\]
but thanks to the second assumption, $V\left( \left\lbrace (x,P) \in \Omega \times \G \, | \, E_0(x,P,V) = + \infty \right\rbrace \right) = 0$. Let 
\[
B = \{ x \in \Omega \: | \: \min_{P \in \G} E_0 (x,P,V) = +\infty \} = \left\lbrace x \in \Omega \: | \: \forall P \in \G, \: E_0 (x,P,V) = +\infty \right\rbrace
\]
then
\begin{align*}
B \times \G = & \left\lbrace (x,P) \in \Omega \times \G \: | \: \forall Q \in \G , \:  E_0 (x,Q,V) = +\infty \right\rbrace \\
\subset & \left\lbrace (x,P) \in \Omega \times \G \: | \:  E_0 (x,P,V) = +\infty \right\rbrace \: .
\end{align*}
Therefore $\V(B) = V(B \times \G) \leq V \left( \left\lbrace (x,P) \in \Omega \times \G \: | \: E_0 (x,P,V) = +\infty \right\rbrace \right) = 0$. So that $\min_{P \in \G} E_0 (x,P,V)$ is finite for $\V$--almost any $x \in \Omega$. And by step $1$, $\V = \theta \cH^d_{| E}$ with $\theta \geq C_1$ for $\cH^d$--almost every $x \in E$, thus for $\cH^d$--almost every $x \in E$,
\begin{equation} \label{pajot_finiteness_jones_numbers}
\int_0^1 \beta_2(x,r,E)^2 \frac{dr}{r} < + \infty \: ,
\end{equation}
and
\begin{equation} \label{pajot_density_condition}
\theta_\ast^d (x,E) = \liminf_{r \to 0_+} \frac{\cH^d (E \cap B_r(x))}{\omega_d r^d} \geq \frac{1}{C_3}  \frac{\V( B_r(x))}{\omega_d r^d} \geq  \frac{C_1}{C_3} > 0 \: .
\end{equation}
\item Step $3$: We need to consider some compact subset of $E$ to apply Pajot's Theorem. The set $E$ being closed in $\Omega$, thus for every compact set $K \subset \Omega$, $E \cap K$ is compact. Thanks to Lemma \ref{lemma_rectif_2}, let $(\omega_n)_n$ be an increasing sequence of relatively compact open sets such that $\Omega = \cup_n \omega_n$ and for all $n$, $\cH^d (E \cap \partial \omega_n) = 0$. Let $K_n = \overline{\omega_n}$, then
\begin{itemize}
\item for all $x \in (E \cap K_n) \setminus \partial K_n = E \cap \overline{\omega_n}$ we have $\theta_\ast^d (x,E\cap K_n) = \theta_\ast^d (x,E)$ and thus by $(\ref{pajot_density_condition})$ and since $\cH^d (E \cap \partial K_n) = 0$,
\begin{equation} \label{pajot_density_condition_local}
\theta_\ast^d (x,E\cap K_n) > 0 \: \text{ for } \cH^d\text{--almost every } x \in E \cap K_n \: ,
\end{equation}
\item thanks to $(\ref{pajot_finiteness_jones_numbers})$, for $\cH^d$--almost every $x \in E \cap K_n$,
\begin{equation} \label{pajot_finiteness_jones_numbers_local}
\int_0^1 \beta_2(x,r,E \cap K_n)^2 \frac{dr}{r} \leq \int_0^1 \beta_2(x,r,E)^2 \frac{dr}{r} < +\infty \: .
\end{equation} 
\end{itemize}
According to $(\ref{pajot_density_condition_local})$ and $(\ref{pajot_finiteness_jones_numbers_local})$, we can apply Pajot's theorem to get the $d$--rectifiability of $E \cap K_n$ for all $n$ and hence the $d$--rectifiability of $E$ and $\V = \theta \cH^d_{| E}$.
\end{enumerate}
Eventually Lemma~\ref{lemma_rectif_3} leads the $d$--rectifiability of the whole varifold $V$.
\end{proof}

\section{The approximation case} \label{approximation_section}

We will now study the approximation case. As we explained before, we introduce some scale parameters (denoted $\alpha_i$ and $\beta_i$) allowing us to consider the approximating objects ``from far enough''. The point is to check that we recover the static conditions (the assumptions $(i)$ and $(ii)$ of Theorem~\ref{static_thm}) in the limit. We begin with some technical lemmas concerning Radon measures. Then we prove a strong property of weak--$\ast$ convergence allowing us to gain some uniformity in the convergence. We end with the proof of the quantitative conditions of rectifiability for varifolds in the approximation case.

\subsection{Some technical tools about Radon measures}

Let us state two technical tools before starting to study the approximation case.

\begin{lemma} \label{mesure_difference_sym_boules}
Let $\Omega \subset \R^n$ be an open set and $(\mu_i)_i$ be a sequence of Radon measures weakly--$\ast$ converging to some Radon measure $\mu$ in $\Omega$. Let $x \in \Omega$ and $x_i \xrightarrow[i \to \infty]{} x$.Then, for every $r > 0$,
\[
\limsup_i \mu_i (B_r(x) \triangle B_r(x_i) ) \leq \mu (\partial B_r (x)) \: .
\]
In particular, if $\mu (\partial B_r (x)) = 0$ then $\mu_i (B_r(x) \triangle B_r(x_i) ) \xrightarrow[i \to \infty]{} 0$.
\end{lemma}

\begin{proof}
Let us define the ring of center $x$ and radii $r_{\text{min}}$ and $r_{\text{max}}$:
\[
R(x,r_{\text{min}},r_{\text{max}}) := \left\lbrace y \in \Omega \, | \, r_{\text{min}} \leq |y-x| \leq r_{\text{max}} \right\rbrace \: .
\]
\begin{minipage}{0.68\textwidth}
It is easy to check that for all $i$, $B_r (x_i) \triangle B_r(x)$ is included into the closed ring of center $x$ and radii $r_{min}^i = r - |x-x_i|$ and $r_{max}^i = r + |x-x_i|$, that is
\[
B_r (x_i) \triangle B_r(x) \subset R(x,r - |x-x_i|,r + |x-x_i|) \: .
\]
Without loss of generality we can assume that $(|x-x_i|)_i$ is decreasing, then the sequence of rings $( R (x , r - |x-x_i | , r + |x- x_i | ))_i$ is decreasing so that for all $p \leq i$,
\begin{align*}
\mu_i ( B_r (x_i) \triangle B_r (x) ) & \leq \mu_i ( R (x , r - |x-x_i | , r + |x- x_i | ) \\
 & \leq \mu_i ( R (x , r - |x-x_p | , r + |x- x_p | ) ) \: .
\end{align*}
\end{minipage}
\begin{minipage}{0.35\textwidth}
\includegraphics[width=0.95\textwidth]{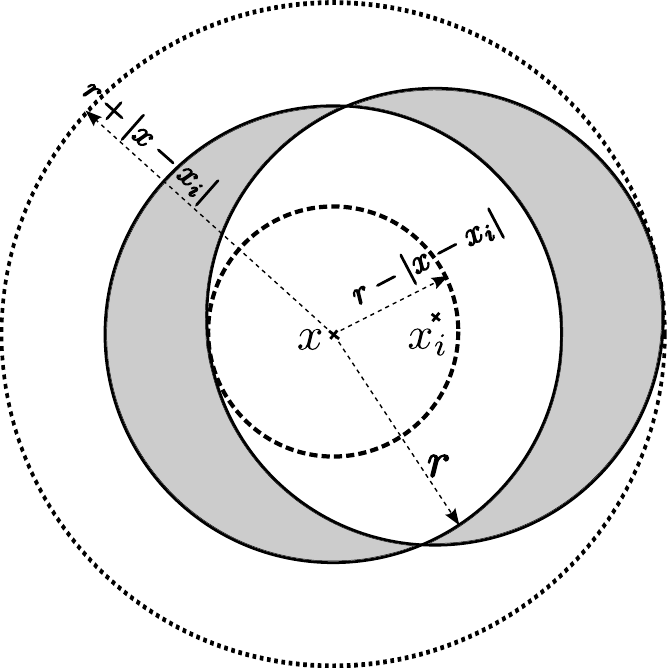}
\end{minipage}
\noindent Consequently, letting $i$ tend to $\infty$ and using the fact that $R (x , r - |x-x_p | , r + |x- x_p | )$ is compact, we have for all $p$,
\[
\limsup_{i \to + \infty} \mu_i ( B_r (x_i) \triangle B_r (x) )  \leq \mu (R (x , r - |x-x_p | , r + |x- x_p | ) )  \: ,
\]
and thus by letting $p \to + \infty$ we finally have,
\[
\limsup_{i \to + \infty} \mu_i ( B_r (x_i) \triangle B_r (x) ) \leq \mu (\partial B_r (x)) \: .
\]
\end{proof}

\begin{prop} \label{weak_star_cv_and_support}
Let $\Omega \subset \R^n$ be an open set and let $(\mu_i)_i$ be a sequence of Radon measures weakly--$\ast$ converging to a Radon measure $\mu$. Then, for every $x \in \supp \mu$, there exist $x_i \in \supp \mu_i$ such that $|x-x_i| \xrightarrow[i \to \infty]{} 0$.
\end{prop}

\begin{proof}
Let $x \in \supp \mu$, and choose $x_i \in \supp \mu_i$ such that $d(x,\supp \mu_i) = |x-x_i|$ (recall that $\supp \mu_i$ is closed). Let us check that $|x-x_i| \xrightarrow[i \to \infty]{} 0$. By contradiction, there exist $\eta > 0$ and a subsequence $(x_{\phi(i)})_i$ such that for all $i$, $| x_{\phi(i)} - x | \geq \eta$.
Therefore, for all $y \in \supp \mu_{\phi(i)}$, $| y -x | \geq | x_{\phi(i)} - x | \geq \eta$ so that
\[
\forall i, \: B_\eta (x) \cap \supp \mu_{\phi(i)} = \emptyset \text{ and thus } \mu_{\phi(i)} \left( B_\eta (x) \right) = 0 \: .
\]
Hence $\mu \left( B_\eta (x) \right) \leq \liminf_i \mu_{\phi(i)} \left( B_\eta (x) \right) = 0 $ and $x \notin \supp \mu$.
\end{proof}

\subsection{Density estimates}

We now look for density estimates for the limit varifold. Indeed, for sets of dimension greater than $d$, for instance $d+1$, the energy $E_0 (x,P,V)$ does not convey information of rectifiability since
\[
\frac{1}{r^{d+1}} \int_{B_r(x)} \left( \frac{d(y-x,P)}{r} \right)^2 \, d\V(y) \leq  \frac{\V(B_r(x))}{r^{d+1}} \leq \theta^{* d+1} (\V,x)
\]
is finite for almost any $x$, not depending on the regularity of $\V$. So that the first assumption in the static theorem (Ahlfors regularity $(\ref{static_density_condition})$ in Theorem~\ref{static_thm}) is quite natural. In this part, we link density estimates on $V_i$ and density estimates on $V$ and then recover the first assumption of the static theorem.

\begin{prop} \label{density_1}
Let $\Omega \subset \R^n$ be an open set. Let $(\mu_i)_i$ be a sequence of Radon measures in $\Omega$, weakly--$\ast$ converging to some Radon measure $\mu$. Assume that there exist $0 < C_1 < C_2$ and a positive decreasing sequence $(\beta_i)_i$ tending to $0$ such that for $\mu_i$--almost every $x \in \Omega$ and for every $r> 0$ such that $\beta_i < r < d(x,\Omega^c)$,
\[
C_1 r^d \leq \mu_i (B_r(x)) \leq C_2 r^d \: .
\]
Then for $\mu$--almost every $x \in \Omega$ and for every $0< r < d(x,\Omega^c)$,
\[
C_1 r^d \leq \mu (B_r(x)) \leq C_2 r^d \: .
\]
\end{prop}

\begin{proof}
Let $\displaystyle A_i = \left\lbrace x \in \Omega \: | \: \forall r \in ]\beta_i, d(x,\Omega^c)[, \: C_1 r^d \leq \mu_i (B_r(x)) \leq C_2 r^d \right\rbrace$.
\begin{enumerate}[(i)]
\item First notice that $A_i$ is closed (thanks to Proposition~\ref{density_2} $(i)$) and $\mu_i (\Omega \setminus A_i) = 0$ so that $\supp \mu_i \subset A_i$.
\item Let $x \in \supp \mu$ and let $0 < r < d(x,\Omega^c)$. By Proposition~\ref{weak_star_cv_and_support}, let $x_i \in \supp \mu_i$ such that $x_i \rightarrow x$ then
\[
\left| \mu_i ( B_r(x)) - \mu_i ( B_r(x_i)) \right| \leq \mu_i \left( B_r(x_i) \triangle B_r(x) \right) \leq \mu_i ( R (x , r - |x-x_i | , r + |x- x_i | ) \: ,
\]
so that by Proposition~\ref{mesure_difference_sym_boules}, $\displaystyle \limsup_i \left| \mu_i ( B_r(x)) - \mu_i ( B_r(x_i)) \right| \leq \mu( B_r(x) )$. Therefore, for almost every $0 < r < d(x,\Omega^c)$, $\displaystyle \mu_i (B_r(x_i)) \xrightarrow[i \to \infty]{} \mu(B_r(x))$. Eventually, as $x_i \in \supp \mu_i \subset A_i$ then for almost every $r < d(x,\Omega^c)$,
\[
C_1 r^d \leq \mu (B_r (x)) = \lim_i \mu_i (B_r (x_i)) \leq C_2 r^d \: .
\]
We can obtain this inequality for all $r$ as in Proposition\ref{density_2}, taking $r^-_k < r < r^+_k$ and $r^-_k, \: r^+_k \rightarrow r$ and such that $\mu (\partial B_{r_k^{+}} (x)) = 0$, $\mu (\partial B_{r_k^{-}} (x)) = 0$.
\end{enumerate}
\end{proof}

\subsection{Uniformity of weak--$\ast$ convergence in some class of functions}

If we try to estimate $E_\alpha (x,P,V_\alpha) - E_\alpha (x,P,V)$, we can have the following:
\begin{align*}
& \left|E_\alpha (x,P,V_\alpha) - E_\alpha (x,P,V) \right| \\
& \leq \frac{1}{\alpha^{d+3}} \int_{r=0}^1  \left| \int_{B_r (x)} d(y-x,P)^2  d \|V_\alpha \| (y) - \int_{B_r (x)} d(y-x,P)^2 d \V (y) \right| \, dr \: .
\end{align*}
We now prove that the integral term tends to $0$ when $V_\alpha \xrightharpoonup[]{\ast} V$. For this purpose, we need a stronger way to write weak--$\ast$ convergence (with some uniformity) using the compactness of some subset of $\xC_c^0 (\Omega)$:
\begin{prop} \label{uniform_weak_star_convergence}
Let $\Omega \subset \R^n$ be an open set and $\left( \mu_i \right)_i$ be a sequence of Radon measures in $\Omega$ weakly--$\ast$ converging to a Radon measure $\mu$. Let $\omega \subset\subset \Omega$ such that $\mu (\partial \omega ) = 0$, then for fixed $k ,C \geq 0$,
\[
\sup \left\lbrace  \left| \int_\omega \phi \, d\mu_i - \int_\omega \phi \, d\mu  \right| \: : \: \phi \in \xLip_k (\omega), \: \| \phi \|_{\infty} \leq C \right\rbrace \xrightarrow[i \to \infty]{} 0
\]
\end{prop}


\begin{proof}
As we already said, the idea is to make use of the compactness of the family $\left\lbrace \phi \in \xLip_k (\omega), \: \| \phi \|_{\infty} \leq C \right\rbrace$. By contradiction, there exists a sequence $(\phi_i)_i$ with $\phi_i \in \xLip_k (\omega)$ and $ \| \phi_i \|_ {\infty} \leq C$ for all $i$ and such that
\[
\left| \int_\omega \phi_i \, d \mu_i - \int_\omega \phi_i \, d\mu \right| \text{ does not converge to } 0 \: .
\]
So that, up to some extraction, there exists $\epsilon > 0$ such that for all $i$,
\[
\left| \int_\omega \phi_i \, d \mu_i - \int_\omega \phi_i \, d\mu \right| > \epsilon \: .
\]
Every $\phi_i$ can be extended to $\phi_i \in \xC (\overline{\omega}) \cap \xLip_k (\overline{\omega})$ and then
\[ \left\lbrace \begin{array}{l}
\left( \phi_i \right)_i \subset \xC (\overline{\omega}) \cap \xLip_k (\overline{\omega}) \rm{\: is \: equilipschitz}, \\
\sup_i \| \phi_i \|_\infty \leq C \: .
\end{array} \right.
\] 
By Ascoli's theorem, up to a subsequence, there exists a function $\phi \in \xC (\overline{\omega}) \cap \xLip_k (\overline{\omega})$ with $\| \phi \|_\infty \leq C$ such that
\[
\phi_i \longrightarrow \phi \rm{\: uniformly \: in \:} \overline{\omega} \: .
\]
We now estimate:
\begin{align*}
\epsilon & < \left| \int_\omega \phi_i \, d \mu_i - \int_\omega \phi_i \, d\mu \right| \\
& \leq \left| \int_\omega \phi_i \, d \mu_i - \int_\omega \phi \, d\mu_i \right| + \left| \int_\omega \phi \, d \mu_i - \int_\omega \phi \, d\mu \right| + \left| \int_\omega \phi \, d \mu - \int_\omega \phi_i \, d\mu \right| \\
& \leq \| \phi_i - \phi \|_{\infty} \, \mu_i (\omega) + \left| \int_\omega \phi \, d \mu_i - \int_\omega \phi \, d\mu \right| + \|\phi - \phi_i \|_{\infty} \, \mu (\omega)
\end{align*}
As $\mu (\partial \omega) = 0$ then $\mu_i(\omega) \xrightarrow[i \to \infty]{} \mu(\omega) < +\infty$ (since $\mu(\omega) \leq \mu(\overline{\omega})$ and $\overline{\omega}$ is compact) so that the first and last terms tend to $0$. Moreover, since $\mu (\partial \omega) = 0$ then for every $f \in \xC^0 (\omega)$ (not necessarily compactly supported),
\[
\int f \, d \mu_i \xrightarrow[i \to \infty]{} \int f \, d \mu \: ,
\] 
which allows to conclude that the second term also tends to $0$ which leads to a contradiction.
\end{proof}

\noindent The following result is the key point of the proof of Theorem~\ref{dynamic_thm}. Let us first define for two Radon measures $\mu$ and $\nu$ in $\Omega$,
\begin{equation} 
\Delta_\omega^{k,C} (\mu , \nu) := \sup \left\lbrace \int_{r=0}^{\frac{d(\overline{\omega},\Omega^c) }{2}} \left| \int_{B_r(x) \cap \omega} \phi \, d\mu - \int_{B_r(x) \cap \omega} \phi \, d\nu \right| dr \: : \: \phi \in \xLip_k (\omega), \: \| \phi \|_{\infty} \leq C , \: x \in \overline{\omega} \right\rbrace \: .
\end{equation}

\begin{prop} \label{key_point}
Let $\Omega \subset \R^n$ be an open set. Let $\left( \mu_i \right)_i$ be a sequence of Radon measures weakly--$\ast$ converging to a Radon measure $\mu$ in $\Omega$ and such that $\sup_i \mu_i (\Omega) < +\infty$. Let $\omega \subset\subset \Omega$ be open such that $\mu (\partial \omega ) = 0$ then, for fixed $k ,C \geq 0$,
\[
\Delta_\omega^{k,C} (\mu_i , \mu) \xrightarrow[i \to +\infty]{} 0 \: .
\]
\end{prop}

\begin{proof}
The upper bound on the radius $r$ ensures that the closure of every considered ball, $B_r (x)$ for $x \in \Omega$, is included in $\Omega$.
We argue as in the proof of Proposition~\ref{uniform_weak_star_convergence}, assuming by contradiction that, after some extraction, there exist a sequence $(\phi_i)_i$ with $\phi_i \in \xLip_k (\omega)$ and $ \| \phi_i \|_ {\infty} \leq C$ for all $i$, and a sequence $(x_i)_i$ with  $x_i \in \overline{\omega}$ for all $i$, and $\epsilon > 0$ such that for all $i$,
\[
\int_{r=0}^{ \frac{d(\overline{\omega},\Omega^c) }{2}} \left| \int_{B_r (x_i) \cap \omega} \phi_i \, d \mu_i - \int_{B_r(x_i) \cap \omega} \phi_i \, d\mu \right| dr > \epsilon \: .
\]
By Ascoli's theorem and up to an extraction, there exist a function $\phi \in \xC^0 (\overline{\omega}) \cap \xLip_k (\overline{\omega})$ with $\| \phi \|_\infty \leq C$ such that $\displaystyle \phi_i \longrightarrow \phi$ uniformly in $\overline{\omega}$. Moreover $\overline{\omega}$ is compact so that, up to another extraction, there exists $ x \in \overline{\omega}$ such that $x_i \longrightarrow x$. We now estimate for every $r$,
\begin{align} 
     & \left| \int_{B_r (x_i) \cap \omega} \phi_i \, d \mu_i - \int_{B_r(x_i) \cap \omega} \phi_i \, d\mu \right| \leq  \left| \int_{B_r (x_i) \cap \omega} \phi_i \, d \mu_i - \int_{B_r(x_i) \cap \omega} \phi \, d\mu_i \right| \nonumber \\
     & + \left| \int_{B_r (x_i)} \phi \, d \mu_i- \int_{B_r(x)} \phi \, d\mu_i \right| + \left| \int_{B_r (x) \cap \omega} \phi \, d \mu_i - \int_{B_r(x) \cap \omega} \phi \, d\mu \right| \nonumber \\
     & + \left| \int_{B_r (x) \cap \omega} \phi \, d \mu - \int_{B_r(x_i) \cap \omega} \phi \, d\mu \right| + \left| \int_{B_r (x_i) \cap \omega} \phi \, d \mu - \int_{B_r(x_i) \cap \omega} \phi_i \, d\mu \right| \nonumber \\
\leq & \| \phi_i - \phi \|_{\infty} \, \mu_i \left( B_r(x_i) \right) \: + \:  \| \phi \|_\infty \, \mu_i \left( B_r (x_i) \triangle B_r(x) \right) \: + \: \left| \int_{B_r ( x) \cap \omega} \phi \, d \mu_i- \int_{B_r( x) \cap \omega} \phi \, d\mu \right| \nonumber \\
     & + \: \|\phi \|_\infty \, \mu \left( B_r (x_i) \triangle B_r(x) \right) \: + \: \| \phi - \phi_i \|_\infty \, \mu (B_r(x_i)) \nonumber \\
\leq & \| \phi_i - \phi \|_\infty \left( \mu_i (\Omega) + \mu(\Omega) \right) + \left| \int_{B_r (x) \cap \omega} \phi \, d \mu_i - \int_{B_r(x) \cap \omega} \phi \, d\mu \right| \label{temp_tend_to_zero} \\
&  + \| \phi \|_\infty \, \left( \mu_i \left( B_r (x_i) \triangle B_r(x) \right) + \mu \left( B_r (x_i) \triangle B_r(x) \right) \right) \: . \nonumber
\end{align}
The first term in the right hand side of $(\ref{temp_tend_to_zero})$ tends to $0$ since $\sup_i \mu_i (\Omega) < +\infty$ also implies $\mu(\Omega) < +\infty$. Concerning the second term, as $\mu (\partial \omega) = 0$ then for all $r \in (0,\frac{d(\overline{\omega},\Omega^c) }{2})$, $\mu (\partial (B_r(x) \cap \omega) ) \leq \mu(\partial B_r(x))$ and therefore the second term tends to $0$ for every $r$ such that $\mu (\partial B_r(x)) = 0$, i.e. for almost every $r \in (0,\frac{d(\overline{\omega},\Omega^c) }{2})$. As for the last term, thanks to Proposition~\ref{mesure_difference_sym_boules} we know that $\displaystyle \limsup_i \mu_i (B_r(x) \triangle B_r(x_i) ) + \mu (B_r(x) \triangle B_r(x_i) ) \leq 2 \mu (\partial B_r (x)) = 0$ for almost every $r \in (0,\frac{d(\overline{\omega},\Omega^c) }{2})$. Moreover the whole quantity $(\ref{temp_tend_to_zero})$ is uniformly bounded by
\[
5 C \left( \mu(\Omega) + \sup_i \mu_i (\Omega) \right) \: .
\]
Consequently the the right hand side of $(\ref{temp_tend_to_zero})$ tends to $0$ for almost every $r \in (0,\frac{d(\overline{\omega},\Omega^c) }{2})$ (such that $\mu(\partial B_r ( x)) = 0$) and is uniformly bounded by the constant $5 C \left( \mu(\Omega) + \sup_j \mu_j (\Omega) \right)$, then by Lebesgue dominated theorem, we have 
\[
\epsilon < \int_{r=0}^{\frac{d(\overline{\omega},\Omega^c) }{2}} \left| \int_{B_r (x_i) \cap \omega} \phi_i \, d \mu_i - \int_{B_r(x_i) \cap \omega} \phi_i \, d\mu \right| dr \xrightarrow[i \to \infty]{} 0
\] which concludes the proof.
\end{proof}

We can now study the convergence of $E_{\alpha_i } (x,P,V_i) - E_{\alpha_i} (x,P,V)$ uniformly with respect to $P$ and locally uniformly with respect to $x$. Indeed, the previous result (Proposition~\ref{key_point}) is given in some compact subset $\overline{\omega} \subset\subset \Omega$. Consequently, we define a local version of our energy:

\begin{dfn}
Let $\Omega \subset \R^n$ be an open set and $\omega \subset\subset \Omega$ be a relatively compact open subset. For every $d$--varifold $V$ in $\Omega$ and for every $x \in \omega$ and $P \in \G$, we define
\[
E_\alpha^\omega (x,P,V) = \int_{r=\alpha}^{\min \left( 1,\frac{ d(\overline{\omega},\Omega^c)  }{2}\right)} \frac{1}{r^d} \int_{B_r (x) \cap \omega} \left( \frac{d (y-x,P)}{r} \right)^2 \, d \V \, \frac{dr}{r}  \:  .
\]
\end{dfn}

\begin{remk}
Notice that
\begin{align*}
E_\alpha^\omega (x,P,V) & = \int_{r=\alpha}^{\min \left( 1,\frac{ d(\overline{\omega},\Omega^c)  }{2}\right)} \frac{1}{r^d} \int_{B_r (x) \cap \omega} \left( \frac{d (y-x,P)}{r} \right)^2 \, d \V \, \frac{dr}{r}  \\
& \leq \int_{r=\alpha}^1 \frac{1}{r^d} \int_{B_r (x) \cap \Omega} \left( \frac{d (y-x,P)}{r} \right)^2 \, d \V \, \frac{dr}{r} = E_\alpha (x,P,V) \: .
\end{align*}
\end{remk}

\begin{prop} \label{alpha_i_choice}
Let $(V_i)_i$ be a sequence of $d$--varifolds weakly---$\ast$ converging to a $d$--varifold $V$ in some open set $\Omega \subset \R^n$ and such that $\sup_i \|V_i \| (\Omega) < +\infty$. For all open subsets $\omega \subset\subset \Omega$ such that $\V (\partial \omega ) = 0$, let us define
\[
\eta_i^\omega := \sup \left\lbrace \left. \int_{r=0}^{\min \left( 1,\frac{ d(\overline{\omega},\Omega^c)  }{2}\right) } \left| \int_{B_r(x) \cap \omega} \phi \, d \|V_i\| - \int_{B_r(x) \cap \omega} \phi \, d \V \right| dr \: \right| \begin{array}{l}
\phi   \in \xLip_{2 (\diam \omega)^2} (\omega), \\
\| \phi \|_{\infty}  \leq (\diam \omega)^2
\end{array} , \: x \in \overline{\omega} \right\rbrace
\]
Then,
\begin{enumerate}
\item for every $0 < \alpha \leq 1$, $\displaystyle \sup_{\substack{ x \in \overline{\omega} \\  P \in \G }} \left| E_\alpha^\omega (x,P,V_i) - E_\alpha^\omega (x,P,V) \right| \leq \frac{\eta_i^\omega}{\alpha^{d+3}}$,
\item $\eta_i^\omega \xrightarrow[i \to \infty]{} 0$
\end{enumerate}
\end{prop}

\begin{proof}
$1.$ is a direct application of Proposition~\ref{key_point}, since $\|V_i\|$ weakly--$\ast$ converges to $\V$. Now let us estimate
\begin{align*}
\left| E_\alpha^\omega \right. & \left. (x,P,V_i)  - E_\alpha^\omega (x,P,V) \right| \\
& \leq \frac{1}{\alpha^{d+3}} \int_{r=0}^{\min \left( 1,\frac{ d(\overline{\omega},\Omega^c)  }{2}\right) } \left| \int_{B_r (x) \cap \omega} d(y-x,P)^2  d \|V_i\| (y) - \int_{B_r (x) \cap \omega} d(y-x,P)^2 d \V (y) \right| \, dr \: .
\end{align*}
For all $x \in \overline{\omega}, \, P \in \G$, let $\varphi_{x,P} (y) :=  d(y-x,P)^2$. One can check that
\begin{enumerate}[$(1)$]
\item $\varphi_{x,P}$ is bounded in $\omega$ by $(\diam \omega)^2$ indeed $\varphi_{x,P} (y) \leq |y - x |^2 \leq (\diam \omega)^2$,
\item $\varphi_{x,P} \in \xLip_{2 (\diam \omega)} (\omega)$ indeed
\begin{align*}
\left| \varphi_{x,P} (y) - \varphi_{x,P} (z) \right| & = \left| d(y-x,P)^2 - d(z-x,P)^2 \right| \\
                                                     & \leq 2 (\diam \omega) \left| d(y-x,P) - d(z-x,P)\right| \\
                                                     & \leq 2 (\diam \omega) \, d(y-z,P) \leq 2 (\diam \omega) \, |y-z| \: .
\end{align*}
\end{enumerate}
Consequently,
\[
\sup_{\substack{ x \in \overline{\omega} \\  P \in \G }} \int_{r=0}^{\min \left( 1,\frac{ d(\overline{\omega},\Omega^c)  }{2}\right) } \left| \int_{B_r (x) \cap \omega} d(y-x,P)^2 d \|V_i\| (y) - \int_{B_r (x) \cap \omega} d(y-x,P)^2 d \V (y) \right| \, dr \leq \eta_i^\omega
\]
and thus,
\[
\sup_{\substack{ x \in \overline{\omega} \\  P \in \G }} \left| E_\alpha^\omega (x,P,V_i) - E_\alpha^\omega (x,P,V) \right| \leq \frac{\eta_i^\omega}{\alpha^{d+3}} \: .
\]
\end{proof}

\noindent It is now easy to deduce the following fact:

\begin{prop} \label{uniform_CV}
Let $(V_i)_i$ be a sequence of $d$--varifolds weakly--$\ast$ converging to a $d$--varifold $V$ in some open set $\Omega \subset \R^n$, and let $\omega \subset\subset \Omega$ be such that $\V (\partial \omega ) = 0$. Assume that $\sup_i \|V_i\|(\Omega) < +\infty$, then, there exists a decreasing sequence $(\alpha_i)_i$ of positive numbers tending to $0$ and such that
\begin{equation} \label{uniform_convergence_1}
\sup_{\substack{ x \in \overline{\omega} \\  P \in \G }} \left| E_{\alpha_i}^\omega (x,P,V_i) - E_{\alpha_i}^\omega (x,P,V) \right|  \xrightarrow[i \to +\infty]{} 0 \: ,
\end{equation}
and for every $x \in \omega$, $P \in \G$, the following pointwise limit holds
\begin{equation} \label{pointwise_convergence_1}
E_0^\omega (x,P,V) = \lim_{i \to \infty} E_{\alpha_i}^\omega (x,P,V_i) \: .
\end{equation}
Conversely, given a decreasing sequence $(\alpha_i)_i$ of positive numbers tending to $0$, there exists an extraction $\phi$ (depending on $\alpha_i$, $V_i$ but independent of $x \in \omega$ and $P \in \G$) such that
\begin{equation}  \label{uniform_convergence_2}
\sup_{\substack{ x \in \overline{\omega} \\  P \in \G }} \left| E_{\alpha_i}^\omega (x,P,V_{\phi(i)}) - E_{\alpha_i}^\omega (x,P,V) \right| \xrightarrow[i \to +\infty]{} 0 \: ,
\end{equation}
and again for every $x \in \omega$, $P \in \G$, the following pointwise limit holds
\begin{equation} \label{pointwise_convergence_2}
E_0^\omega (x,P,V) = \lim_{i \to \infty} E_{\alpha_i}^\omega (x,P,V_{\phi(i)}) \: .
\end{equation}
\end{prop}

\begin{proof}
Thanks to Proposition~\ref{alpha_i_choice}, for every $\alpha > 0$,
\[
\sup_{\substack{ x \in \overline{\omega} \\  P \in \G }} \left| E_\alpha^\omega (x,P,V_i) - E_\alpha^\omega (x,P,V) \right| \leq \frac{\eta_i^\omega}{\alpha^{d+3}} \: \text{ and } \eta_i^\omega \xrightarrow[i \to \infty]{} 0 \: ,
\]
hence we can choose $(\alpha_i)_i$ such that $\displaystyle \frac{\eta_i^\omega}{\alpha_i^{d+3}} \xrightarrow[i \to \infty]{} 0$.
Conversely, given the sequence $(\alpha_i)_i$ tending to $0$, we can extract a subsequence $(\eta_{\phi(i)}^\omega)_i$ such that $\displaystyle \frac{\eta_{\phi(i)}^\omega}{\alpha_i^{d+3}} \xrightarrow[i \to \infty]{} 0$. For fixed $x \in \omega$ and $P \in \G$, the pointwise convergences to the averaged height excess energy $E_0^\omega$,  $(\ref{pointwise_convergence_1})$ and $(\ref{pointwise_convergence_2})$, are a consequence of the previous convergence properties $(\ref{uniform_convergence_1})$ and $(\ref{uniform_convergence_2})$, and of the monotone convergence $E_\alpha^\omega (x,P,V) \xrightarrow[\alpha \to 0]{} E_0^\omega (x,P,V)$.
\end{proof}

Now, we can use this uniform convergence result in $\omega \times \G$ to deduce the convergence of the integrated energies.

\begin{prop} \label{convergence_integrated_energies}
Let $\Omega \subset \R^n$ be an open set and let $(V_i)_i$ be a sequence of $d$--varifolds in $\Omega$ weakly--$\ast$ converging to some $d$--varifold $V$ and such that $\sup_i \| V_i \| (\Omega) < +\infty$.  Fix a decreasing sequence $(\alpha_i)_i$ of positive numbers tending to $0$. Let $\omega \subset \subset \Omega$ with $\V(\partial \omega) =0$. Then there exists an extraction $\psi$ such that
\[
\int_{\omega \times \G} E_0^\omega (x,P,V) \, d V(x,P) = \lim_{i \to \infty}  \int_{\omega \times \G} E_{\alpha_i}^\omega (x,P,V_{\psi(i)}) \, d V_{\psi(i)}(x,P)  \: . 
\]
\end{prop}

\begin{proof}
\begin{enumerate}[$-$]
\item Step $1$: Let $(\alpha_i)_i \downarrow 0$ and $V_i \xrightharpoonup[i \to \infty]{\ast} V$. Thanks to Proposition~\ref{uniform_CV}), there exists an extraction $\phi$ such that
\[
\sup_{\substack{ x \in \overline{\omega} \\  P \in \G }} \left| E_{\alpha_i}^\omega (x,P,V_{\phi(i)}) - E_{\alpha_i}^\omega (x,P,V) \right| \xrightarrow[i \to \infty]{} 0 \: .
\]
But $\sup_i V_{\phi(i)} (\omega \times \G) \leq \sup_i \|V_i\|(\Omega) < +\infty$, hence
\begin{equation} \label{energy_convergence_1}
\left| \int_{\omega \times \G} E_{\alpha_i}^\omega (x,P,V_{\phi(i)}) \, d V_{\phi(i)} (x, P) - \int_{\omega \times \G} E_{\alpha_i}^\omega (x,P,V) \, d V_{\phi(i)} (x,P) \right| \xrightarrow[i \to \infty]{} 0 \: .
\end{equation}
\item Step $2$: Now, we estimate
\begin{align*}
& \left| \int_{\omega \times \G} E_{\alpha_i}^\omega (x,P,V) \, d V_{\phi(i)} (x, P) - \int_{\omega \times \G} E_{\alpha_i}^\omega (x,P,V) \, d V (x,P) \right| \\
& \leq \int_{r=\alpha_i}^{\min \left( 1 , \frac{d(\overline{\omega}, \Omega^c)}{2} \right)}  \frac{1}{r^{d+1}} \left| \int_{\omega \times \G} \int_{B_r(x) \cap \omega} \left( \frac{d(y-x,P)}{r} \right)^2 \, d \V(y) \, d V_{\phi(i)} (x,P) \right. \\
& \left. - \int_{\omega \times \G} \int_{B_r(x) \cap \omega} \left( \frac{d(y-x,P)}{r} \right)^2 \, d \V(y) \, d V (x,P) \right|  \, dr \\
& \leq \frac{1}{\alpha_i^{d+3}} \int_{r=\alpha_i}^{\min \left( 1 , \frac{d(\overline{\omega}, \Omega^c)}{2} \right)}   \left| \int_{\omega \times \G} g_r(x,P) d V_{\phi(i)} (x,P) - \int_{\omega \times \G} g_r (x,P) \, d V (x,P) \right|  \, dr \: ,
\end{align*}
with $g_r(x,P) = \displaystyle \int_{B_r(x) \cap \omega} d(y-x,P)^2 \, d \V(y)$. For every $\displaystyle r < \min \left( 1 , \frac{d(\overline{\omega}, \Omega^c)}{2} \right)$, $g_r$ is bounded by $1$. Moreover the set of discontinuities of $g_r$, denoted by $\text{disc} (g_r)$, satisfies
\begin{align*}
\text{disc} (g_r) & \subset \left\lbrace (x,P) \in \omega \times \G \: : \: \V (\partial (B_r(x) \cap \omega)) > 0 \right\rbrace \\
& \subset \left\lbrace (x,P) \in \omega \times \G \: : \: \V (\partial B_r(x)) > 0 \right\rbrace \: .
\end{align*}
Hence $V (\text{disc} (g_r)) \leq \V \left( \left\lbrace x \in \omega \: : \: \V (\partial B_r(x)) > 0 \right\rbrace \right) = 0$ for almost every $r$ by Proposition~\ref{boundary_weight}. Consequently,
\[
\left| \int_{\omega \times \G} g_r(x,P) d V_{\phi(i)} (x,P) - \int_{\omega \times \G} g_r (x,P) \, d V (x,P) \right| \xrightarrow[i \to \infty]{} 0 \quad \text{for a.e. } r \: ,
\]
and then by dominated convergence,
\[
\int_{r=0}^{\min \left( 1 , \frac{d(\overline{\omega}, \Omega^c)}{2} \right)}   \left| \int_{\omega \times \G} g_r(x,P) d V_{\phi(i)} (x,P) - \int_{\omega \times \G} g_r (x,P) \, d V (x,P) \right|  \, dr \xrightarrow[i \to \infty]{} 0 \: .
\]
It is then possible to extract, again, a subsequence $(V_{\psi(i)})_i$ such that
\begin{equation} \label{energy_convergence_2}
\left| \int_{\omega \times \G} E_{\alpha_i}^\omega (x,P,V) \, d V_{\psi(i)} (x, P) - \int_{\omega \times \G} E_{\alpha_i}^\omega (x,P,V) \, d V (x,P) \right| \xrightarrow[i \to \infty]{} 0 \: .
\end{equation}
\item Step $3$: Eventually by $(\ref{energy_convergence_1})$, $(\ref{energy_convergence_2})$ and monotone convergence, there exists an extraction $\psi$ such that
\begin{align*}
\int_{\omega \times \G} E_0^\omega (x,P,V) \, dV(x,P) &  = \lim_{i \to \infty} \int_{\omega \times \G} E_{\alpha_i}^\omega (x,P,V) \, dV(x,P) \\
& = \lim_{i \to \infty} \int_{\omega \times \G} E_{\alpha_i}^\omega (x,P,V_{\psi(i)}) \, dV_{\psi(i)}(x,P) \: .
\end{align*}
\end{enumerate}
\end{proof}


\subsection{Rectifiability theorem}

We can now state the main result.

\begin{theo}\label{dynamic_theorem}
Let $\Omega \subset \R^n$ be an open set and let $(V_i)_i$ be a sequence of $d$--varifolds in $\Omega$ weakly--$\ast$ converging to some $d$--varifold $V$ and such that $\sup_i \| V_i \| (\Omega) < +\infty$.  Fix $(\alpha_i)_i$ and $(\beta_i)_i$ decreasing sequences of positive numbers tending to $0$ and assume that:
\begin{enumerate}[(i)]
\item there exist $0 < C_1 < C_2$ such that for $\| V_i \|$--almost every $x \in \Omega$ and for every $\beta_i < r < d(x,\Omega^c)$,
\begin{equation} \label{dynamic_density_condition}
C_1 \omega_d r^d \leq \| V_i \| (B_r(x)) \leq C_2 \omega_d r^d \: ,
\end{equation}
\item \begin{equation} \label{dynamic_jones_condition}
\sup_i \int_{\Omega \times \G} E_{\alpha_i} (x,P,V_i) \, d V_i(x,P) < +\infty \: .
\end{equation}
\end{enumerate}
Then $V$ is a rectifiable $d$--varifold.
\end{theo}

\begin{proof}
The point is to see that these two assumptions $(\ref{dynamic_density_condition})$ and $(\ref{dynamic_jones_condition})$ actually imply the assumptions of the static theorem (Theorem~\ref{static_theorem}) for the limit varifold $V$.
\begin{enumerate}[$-$]
\item Step $1$: The first assumption $(\ref{dynamic_density_condition})$ and Proposition~\ref{density_1} lead to the first assumption of the static theorem: there exist $0 < C_1 < C_2$ such that for $\V$--almost every $x \in \Omega$ and for every $0 < r < d(x,\Omega^c)$,
\[
C_1 \omega_d r^d \leq \| V \| (B_r(x)) \leq C_2 \omega_d r^d \: .
\]
\item Step $2$: Let $\omega \subset\subset \Omega$ be a relatively compact open subset such that $\V(\partial \omega) = 0$ then, thanks to Proposition~\ref{convergence_integrated_energies}, we know that there exists some extraction $\phi$ such that
\begin{equation} \label{integrated_energy_convergence_1}
\int_{\omega \times \G} E_0^\omega (x,P,V) \, d V(x,P) = \lim_{i \to \infty}  \int_{\omega \times \G} E_{\alpha_i}^\omega (x,P,V_{\phi(i)}) \, d V_{\phi(i)}(x,P)  \: . 
\end{equation}
But $E_\alpha^\omega$ is decreasing in $\alpha$ and $\alpha_{\phi(i)} \leq  \alpha_i$, therefore for every $(x,P) \in \omega \times \G$,
\[
E_{\alpha_i}^\omega (x,P,V_{\phi(i)}) \leq E_{\alpha_{\phi(i)}}^\omega (x,P,V_{\phi(i)}) \: ,
\]
hence
\begin{equation} \label{integrated_energy_convergence_2}
\sup_i \int_{\omega \times \G} E_{\alpha_i}^\omega (x,P,V_{\phi(i)}) \, d V_{\phi(i)}(x,P) \leq \sup_i \int_{\omega \times \G} E_{\alpha_{\phi(i)}}^\omega (x,P,V_{\phi(i)}) \, d V_{\phi(i)}(x,P)  \: .
\end{equation}
Moreover, recall that $E_{\alpha_i}^\omega (x , P,V_i) \leq E_{\alpha_i} (x , P,V_i) $ and thus
\begin{equation} \label{integrated_energy_convergence_3}
\sup_i \int_{\omega \times \G} E_{\alpha_i}^\omega (x , P,V_i) \, d V_i(x,P) \leq \sup_i \int_{\Omega \times \G} E_{\alpha_i} (x , P,V_i) \, d V_i(x,P) \leq C \: .
\end{equation}
Eventually, by $(\ref{integrated_energy_convergence_1})$, $(\ref{integrated_energy_convergence_2})$ and $(\ref{integrated_energy_convergence_3})$,
\begin{equation} \label{integrated_energy_finiteness}
\int_{\omega \times \G} E_0^\omega (x,P,V) \, d V(x,P) \leq C  \: . 
\end{equation}
\item Step $3$: By $(\ref{integrated_energy_finiteness})$, for every $\omega \subset\subset \Omega$ such that $\V(\partial \omega) = 0$ we get that 
\[
V \left( \left\lbrace (x,P) \in \omega \times \G \: | \: E_0^\omega (x,P,V) = +\infty \right\rbrace \right) = 0 \: .
\]
At the same time, for $x \in \omega$ and $P \in \G$,
\begin{align*}
\left| E_0 (x,P,V)\right. & \left.  - E_0^\omega (x,P,V) \right| = \int_{r=\min \left( 1 , \frac{d(\overline{\omega},\Omega^c)}{2} \right)}^1 \frac{1}{r^{d+1}} \int_{B_r(x) \cap \Omega} \left( \frac{d(y-x,P)}{r} \right)^2 \, d\V(y) \, dr \\
& + \int_{r=0}^{\min \left( 1 , \frac{d(\overline{\omega},\Omega^c)}{2} \right)} \frac{1}{r^{d+1}} \int_{B_r(x) \cap (\Omega \setminus \overline{\omega})} \left( \frac{d(y-x,P)}{r} \right)^2 \, d\V(y) \, dr\\
& \leq \left( \frac{2}{d(\overline{\omega},\Omega^c)} \right)^{d+1} \V(\Omega) + \int_{r=d(x,\omega^c)}^{\min \left( 1 , \frac{d(\overline{\omega},\Omega^c)}{2} \right)} \frac{1}{r^{d+1}} \int_{B_r(x) \cap \Omega } \left( \frac{d(y-x,P)}{r} \right)^2 \, d\V(y) \, dr\\
 & \leq  \left( \left( \frac{2}{d(\overline{\omega},\Omega^c)} \right)^{d+1} + \left( \frac{1}{d(x,\omega^c)} \right)^{d+1} \right) \V(\Omega) < + \infty \: .
\end{align*}
Hence $E_0^\omega (x,P,V) = +\infty$ if and only if  $E_0 (x,P,V)= +\infty$, and consequently,
\[
V \left( \left\lbrace (x,P) \in \omega \times \G \: | \: E_0 (x,P,V) = +\infty \right\rbrace \right) = V \left( \left\lbrace (x,P) \in \omega \times \G \: | \: E_0^\omega (x,P,V) = +\infty \right\rbrace \right) \: .
\]
Now, thanks to Lemma~\ref{lemma_rectif_2}, we decompose $\Omega$ into $\Omega = \cup_k \omega_k$ with $\forall k$, $\omega_{k+1} \subset\subset \omega_k \subset\subset \Omega$ and $\V (\partial \omega_k)=0$. Then
\begin{align*}
V \left( \left\lbrace (x,P) \in \Omega \times \G \: | \: E_0 (x,P,V) = +\infty \right\rbrace \right) & = \lim_k V \left( \left\lbrace (x,P) \in \omega_k \times \G \: | \: E_0 (x,P,V) = +\infty \right\rbrace \right) \\
& = \lim_k V \left( \left\lbrace (x,P) \in \omega_k \times \G \: | \: E_0^{\omega_k} (x,P,V) = +\infty \right\rbrace \right)\\
& = 0 \: .
\end{align*}
\end{enumerate}
Applying the static theorem (Theorem~\ref{static_thm}) allows us to conclude the proof. 
\end{proof}

\medskip
In Theorem~\ref{dynamic_thm}, we have found conditions $(\ref{dynamic_density_condition})$ and $(\ref{dynamic_jones_condition})$ ensuring the rectifiability of the weak--$\ast$ limit $V$ of a sequence of $d$--varifolds $(V_i)_i$. Recall that the condition
\begin{equation} \label{allard_first_variation_condition}
\sup_i |\delta V_i|(\Omega) < +\infty
\end{equation}
together with the condition $(\ref{dynamic_density_condition})$ also ensure the rectifiability of the weak--$\ast$ limit $V$ of $(V_i)_i$. But, in Proposition~\ref{computation_first_variation_discrete_varifold}, we have computed the first variation of a discrete varifold (discrete varifolds are defined in Example~\ref{diffuse_discrete_varifolds}) and we have seen in Example~\ref{first_variation_explosion} that even in the case where the limit varifold $V$ is very simple (we considered a straight line), the natural approximations of $V$ by discrete varifolds $V_i$ generally do not satisfy $(\ref{allard_first_variation_condition})$ even though $|\delta V|(\Omega) = 0$.

\noindent We now check that the condition $(\ref{dynamic_jones_condition})$ in Theorem~\ref{dynamic_thm} is better adapted to general sequences of varifolds than the control of the first variation $(\ref{allard_first_variation_condition})$. Indeed, in the next Proposition, we prove that given a $d$--varifold $V$ with some regularity property, and given any sequence of $d$--varifolds $V_i \xrightharpoonup[i \to \infty]{\ast} V$, there exists a subsequence of $(V_i)_i$ satisfying a local version of condition $(\ref{dynamic_jones_condition})$ in Theorem~\ref{dynamic_thm}.

\begin{prop} \label{local_jones_uniform_condition}
Let $\Omega \subset \R^n$ be an open set and let $V$ be a $d$--varifold in $\Omega$ such that
\[
\int_{\Omega \times \G} E_0 (x,P,V) \, dV(x,P) < + \infty \: .
\]
Let $(V_i)_i$ be a sequence of $d$--varifolds weakly--$\ast$ converging to $V$ with $\sup_i \|V_i\|(\Omega) < +\infty$. Then, given $\alpha_i \downarrow 0$, for every $\omega \subset\subset \Omega$ such that $\V (\partial \omega) = 0$, there exists a subsequence $(W_i)_i = (V_{\phi(i)})_i$ such that
\begin{equation} \label{local_jones_condition}
\sup_i \int_{\omega \times \G} E_{\alpha_i}^\omega (x,P,W_i) \, dW_i (x,P) < + \infty \: .
\end{equation}
\end{prop}

\begin{proof}
It is a direct consequence of Proposition~\ref{convergence_integrated_energies}.
\end{proof}



The condition $(\ref{local_jones_condition})$ is expressed in terms of the local version $E_\alpha^\omega$ of $E_\alpha$. In the case where the varifolds are contained in the same compact set, then global condition $(\ref{dynamic_jones_condition})$ of Theorem~\ref{dynamic_thm} is satisfied by some subsequence.

\begin{prop}  \label{global_jones_uniform_condition}
Let $\alpha_i \downarrow 0$. Let $V$ be a rectifiable $d$--varifold in $\R^n$ with compact support and such that
\[
\int_{\omega \times \G} E_0 (x,P,V) \, dV(x,P) < + \infty \: .
\]
Assume moreover that there exists some sequence of $d$--varifolds $(V_i)_i$ weakly--$\ast$ converging to $V$ with $\sup_i \|V_i\|(\R^n) < +\infty$. Then for any $\omega \subset\subset \R^n$ such that $\supp \V + B_1(0) \subset \omega$ and for all $i$, $\supp \|V_i\| + B_1(0) \subset \omega$, there exists a subsequence $(V_{\phi(i)})_i$ such that
\[
\sup_i \int_{\omega \times \G} E_{\alpha_i} (x,P,V_{\phi(i)}) \, dV_{\phi(i)} (x,P) < + \infty \: .
\]
\end{prop}

\begin{proof}
It is again a direct consequence of Proposition~\ref{convergence_integrated_energies} (since $\overline{\omega}$ is compact and $\V (\partial \omega) = 0$) combined with the fact that $\supp \V + B_1(0) \subset \omega$ implies 
\[
E_\alpha^\omega (x,P,V) = \int_{r=\alpha}^{\min \left( 1,\frac{ d(\overline{\omega},(\R^n)^c)  }{2}\right)} \frac{1}{r^d} \int_{B_r (x) \cap \omega} \left( \frac{d (y-x,P)}{r} \right)^2 \, d \V \, \frac{dr}{r} = E_\alpha (x,P,V) \: .
\]
\end{proof}

Given $V_i \xrightharpoonup[i \to +\infty]{\: \ast \:} V$ and $\alpha_i \downarrow 0$, the previous propositions~$\ref{local_jones_uniform_condition}$ and $\ref{global_jones_uniform_condition}$ give a subsequence $(V_{\phi(i)})_i$ satisfying $(\ref{dynamic_jones_condition})$
\[
\sup_i \int E_{\alpha_i} (x,P,V_{\phi(i)}) \, d V_{\phi(i)}(x,P) < +\infty \:
\]
In the following proposition, we focus on sequences of discrete varifolds defined in Example~\ref{diffuse_discrete_varifolds}. Under some uniform regularity assumption on $V$, we give a sequence $(V_i)_i$ of discrete varifolds such that
\[
V_i \xrightharpoonup[i \to +\infty]{\: \ast \:} V \: ,
\]
and a condition linking the scale parameter $\alpha_i$ and the size $\delta_i$ of the mesh associated to the discrete varifold $V_i$, ensuring that $(\ref{dynamic_jones_condition})$ holds for $V_i$ and not for a subsequence.

\begin{theo}
Let $V = v(M,\theta)$ be a rectifiable $d$--varifold in $\R^n$ with finite mass $\V(\Omega)<+\infty$ and compact support. Let $\delta_i \downarrow 0$ be a sequence of infinitesimals and $(\cK_i)_i$ a sequence of meshes satisfying
\[
\sup_{K \in \cK_i} \diam (K) \leq \delta_i \xrightarrow[i \to + \infty]{} 0 \: .
\]
Assume that there exists $0<\beta< 1$ and $C>0$ such that for $\V$--almost every $x$, $y \in \Omega$,
\[
\| T_x M - T_y M \| \leq C | x - y |^\beta \: .
\]
Define the sequence of discrete varifolds:
\[
V_i = \sum_{K \in \cK_i} \frac{m_K^i}{|K|} \cL^n \otimes \delta_{P_K^i} \text{ with } m_K^i = \V(K) \text{ and } P_K^i \in \argmin_{P \in \G} \int_{K \times \G} \left\| P - S \right\| \, dV(x,S) \: .
\]
Then,
\begin{enumerate}[$(i)$]
\item $V_i \xrightharpoonup[i \to + \infty]{ \: \ast \: } V$,
\item For any sequence of infinitesimals $\alpha_i \downarrow 0$ and such that for all $i$,
\begin{equation}
\frac{\delta_i^\beta}{\alpha_i^{d+3}} \xrightarrow[i \to +\infty]{} 0 \: , \label{quantitative_condition_on_scale}
\end{equation}
we have,
\[
\int_{\R^n \times \G} E_0 (x,P,V) \, dV(x,P) = \lim_{i \to +\infty} \int_{\R^n \times \G} E_{\alpha_i} (x,P,V_i) \, dV_i(x,P) < + \infty \: .
\]
\end{enumerate}
\end{theo}

\begin{remk}
We insist on the fact that the condition on the scale parameters $\alpha_i$ and the size of the mesh $\delta_i$ is not dependent on $V_i$ but only on the regularity of $V$ i.e. on $\beta$ (and on the dimension $d$). 
\end{remk}

\begin{proof}
\begin{enumerate}[$-$]
\item Step $1$: Let $\phi \in \xLip(\R^n \times \G)$ with Lipschitz constant $\xlip (\phi)$, then
\[
\left| \left\langle V_i , \phi \right\rangle - \left\langle V , \phi \right\rangle \right| \leq \delta_i \xlip (\phi) \V(\R^n) + \xlip(\phi) \, \int_{\R^n \times \G} \left\| P^i (y) - T \right\| \, dV(y,T) \: ,
\]
where $P^i : \R^n \rightarrow \G$ is cell-wise constant, and $P^i = P^i_K$ in $K$.

\noindent Indeed,
\begin{align}
& \left| \left\langle V_i , \phi \right\rangle  - \left\langle V , \phi \right\rangle \right|  = \left| \int_{\R^n \times \G} \phi (x,S) \, dV_i(x,S) - \int_{\R^n \times \G} \phi(y,T) \, dV(y,T) \right| \nonumber \\
& = \left| \sum_{K \in \cK_i} \int_K \phi(x,P_K^i) \frac{\V(K)}{|K|} \, d \cL^n (x) - \sum_{K \in \cK_i} \int_{K \times \G} \phi(y,T) \, dV(y,T) \right| \nonumber \\
& = \left| \sum_{K \in \cK_i} \int_{x \in K} \int_{K \times \G} \phi(x,P_K^i) \, dV(y,T) \, \frac{d \cL^n(x)}{|K|}  - \sum_{K \in \cK_i} \int_{x \in K} \int_{K \times \G} \phi(y,T) \, dV(y,T) \, \frac{d \cL^n(x)}{|K|} \right| \nonumber  \\
& \leq \sum_{K \in \cK_i} \int_{x \in K} \int_{(y,T) \in K \times \G} \underbrace{\left| \phi(x,P_K^i) - \phi(y,T) \right| }_{\leq \xlip(\phi) \left( |x-y| + \left\| P_K^i - T \right\| \right) } \, dV(y,T) \, \frac{d \cL^n(x)}{|K|}  \label{discrete_varifolds_lipschitz_case} \\
& \leq \delta_i \xlip(\phi) \V(\R^n) + \xlip(\phi) \, \int_{\R^n \times \G} \left\| P^i (y) - T \right\| \, dV(y,T) \: . \label{quantitative_lipschitz_control}
\end{align}
We now study the convergence of the term $\displaystyle \int_{\R^n \times \G} \left\| P^i (y) - T \right\| \, dV(y,T)$.

\item Step $2$:
\begin{equation}
\int_{\R^n \times \G} \left\| P^i (y) - T \right\| \, dV(y,T) \leq 2 C \delta_i^\beta \V (\R^n) \xrightarrow[i \to +\infty]{} 0 \: . \label{discrete_varifold_convergence_tangent_plane}
\end{equation}

\noindent First define, for all $i$, $A^i : \R^n \rightarrow \cM_n(\R)$ cell-wise constant: 
\[
A^i = A^i_K = \frac{1}{\V(K)} \int_K T_u M \, d\V(u) \text{ constant in the cell } K \in \cK_i \, .
\]
Then,
\begin{align*}
\int_{\R^n \times \G} \left\| A^i (y) - T \right\| \, d V(y,T) & = \sum_{K \in \cK_i} \int_K \left\|  \frac{1}{\V(K)} \int_K T_u M \, d\V(u) - T_y M  \right\| \, d \V(y) \\
& \leq \sum_{K \in \cK_i} \int_K   \frac{1}{\V(K)} \int_K \underbrace{ \left\| T_u M  - T_y M  \right\| }_{\leq C |u-y|^\beta \leq C \delta_i^\beta} \, d\V(u) \, d \V(y) \\
& \leq C \delta_i^\beta \V (\R^n) \: .
\end{align*}
Consequently, $\displaystyle \int_K  \left\| A^i (y) - T_y M \right\| \, d \V(y) = \epsilon_K^i$ with $\displaystyle \sum_{K \in \cK_i} \epsilon_K^i < C \delta_i^\beta \V (\R^n)$. In particular, for all $K \in \cK_i$, there exists $y_K \in K$ such that
\[
\left\| A^i (y_K) - T_{y_K} M \right\| \leq \frac{\epsilon_K^i}{\V(K)} \: .
\]
Define $T^i : \R^n \rightarrow \G$, constant in each cell, by $T^i (y) = T^i_K = T_{y_K} M$ for $K \in \cK_i$ and $y \in K$, and then,
\begin{align*}
\int_{\R^n \times \G} \left\| T^i (y) - T \right\| \, dV(y,T) & = \sum_{K \in \cK_i} \int_K  \left\| T_{y_K} M  - T_y M \right\| \, d \V(y) \\
& \leq \sum_{K \in \cK_i} \int_K  \| T_{y_K} M  - \underbrace{A^i(y)}_{= A^i(y_K)}\| \, d \V(y) + \int_{\R^n \times \G} \left\| A^i (y) - T \right\| \, dV(y,T) \\
& \leq \sum_{K \in \cK_i} \int_K \frac{\epsilon_K^i}{\V(K)} d \V(y) + C \delta_i^\beta \V (\R^n) \\
& \leq 2 C \delta_i^\beta \V (\R^n) \: .
\end{align*}
Now remind that for all $K \in \cK_i$, $\displaystyle P_K^i \in \argmin_{P \in \G} \int_{K \times \G} \left\| P - T \right\| \, dV(y,T)$ so that,
\begin{align*}
\int_{\R^n \times \G} \left\| P^i (y) - T \right\| \, dV(y,T) & = \sum_{K \in \cK_i} \int_{K \times \G}  \left\| P^i_K  - T \right\| \, dV(y,T) \\
& \leq \sum_{K \in \cK_i} \int_{K \times \G}  \left\| T^i_K  - T \right\| \, dV(y,T) \\
& \leq 2 C \delta_i^\beta \V (\R^n) 
\end{align*}

\item Step $3$: $V_i \xrightarrow[i \to + \infty]{\: \ast \:} V $.

\noindent Thanks to Steps $1$ and $2$, we have proved that for any $\phi \in \xLip (\Omega \times \G)$,
\begin{equation} \label{convergence_diffuse_discrete_varifold_lipschitz_case}
\left\langle V_i , \phi \right\rangle \xrightarrow[i \to +\infty]{} \left\langle V , \phi \right\rangle \: ,
\end{equation}
it remains to check the case $\phi \in \xC_c^0(\R^n \times \G)$. Let $\phi \in \xC_c^0 (\R^n \times \G)$ and $\epsilon > 0$. We can extend $\phi$ into $\overline{\phi} \in \xC_c^0 (\R^n \times \cM_n(\R))$ by Tietze-Urysohn theorem since $\G$ is closed. Then, by density of $\xLip (\R^n \times \cM_n(\R))$ in $\xC_c^0 (\R^n \times \cM_n(\R))$ with respect to the uniform topology, there exists $\overline{\psi} \in \xLip (\R^n \times \cM_n(\R))$ such that $\left\| \overline{\phi} - \overline{\psi} \right\|_{\infty} < \epsilon$. Let now $\psi \in \xLip (\R^n \times \G)$ be the restriction of $\overline{\psi}$ to $\R^n \times \G$, then,
\begin{align*}
\left| \left\langle V, \phi \right\rangle - \left\langle V_i , \phi \right\rangle \right| & \leq \left| \left\langle V, \phi \right\rangle - \left\langle V , \psi \right\rangle \right| + \left| \left\langle V, \psi \right\rangle - \left\langle V_i , \psi \right\rangle \right| + \left| \left\langle V_i , \psi \right\rangle - \left\langle V_i , \phi \right\rangle \right|\\
& \leq \V (\R^n) \| \phi - \psi \|_{\infty} +  \left| \left\langle V, \psi \right\rangle - \left\langle V_i , \psi \right\rangle \right| + \| V_i \| ( \R^n) \| \phi - \psi \|_{\infty} \: .
\end{align*}
As $\| V_i \|(\R^n) = \V(\R^n)$ for all $i$ by definition of $V_i$ and $\left| \left\langle V, \psi \right\rangle - \left\langle V_i , \psi \right\rangle \right| \xrightarrow[i \to + \infty]{} 0$ by $(\ref{convergence_diffuse_discrete_varifold_lipschitz_case})$, there exists $i$ large enough such that
\[
\left| \left\langle V, \phi \right\rangle - \left\langle V_i , \phi \right\rangle \right| \leq \left( 2 \V(\R^n) + 1 \right) \, \epsilon \: ,
\]
which concludes the proof of the weak--$\ast$ convergence.

\noindent We now estimate,
\begin{align}
\left| \int_{\R^n \times \G} \right. & \left. E_\alpha (x,P,V) \, d V(x,P) - \int_{\R^n \times \G} E_\alpha (x,P,V_i) \, d V_i(x,P) \right| \\
 & \leq \left| \int_{\R^n \times \G} E_\alpha (x,P,V) \, d V(x,P) - \int_{\R^n \times \G} E_\alpha (x,P,V) \, d V_i (x,P) \right| \label{energy_difference_1} \\
 & + \left| \int_{\R^n \times \G} E_\alpha (x,P,V) \, d V_i (x,P) - \int_{\R^n \times \G} E_\alpha (x,P,V_i) \, d V_i(x,P) \right| \label{energy_difference_2}
\end{align}

\item Step $4$: We begin with $(\ref{energy_difference_1})$ and we prove that
\begin{equation} \label{energy_difference_4}
\left| \int_{\R^n \times \G} E_\alpha (x,P,V) \, d V_i (x,P) - \int_{\R^n \times \G} E_\alpha (x,P,V) \, d V(x,P) \right| \leq \frac{1}{\alpha^{d+3}}  \V(\R^n)^2 \left[ 4 \delta_i + 2C \delta_i^\beta \right] \: .
\end{equation}
\begin{align}
 \left| \int_{\R^n \times \G} \right. & \left. E_\alpha (x,P,V) \, d V_i (x,P) - \int_{\R^n \times \G} E_\alpha (x,P,V) \, d V(x,P) \right| \nonumber \\
& = \left| \int_{\R^n \times \G} \int_{r = \alpha}^1 \frac{1}{r^{d+1}} \int_{y \in B_r(x)} \left( \frac{d(y-x,P)}{r} \right)^2 \, d \V(y) \, dr \, dV_i (x,P) \right. \nonumber  \\
& \left. - \int_{\R^n \times \G} \int_{r = \alpha}^1 \frac{1}{r^{d+1}} \int_{y \in B_r(x)} \left( \frac{d(y-x,P)}{r} \right)^2 \, d \V(y) \, dr \, dV (x,P) \right| \nonumber  \\
& \leq \int_{r = \alpha}^1 \frac{1}{r^{d+3}} \int_{y \in \R^n} \left| \int_{\R^n \times \G} \one_{ \{ |y-x|<r \} } (x) \left( d(y-x,P) \right)^2  \, dV_i (x,P) \right. \nonumber  \\
& \left. - \int_{\R^n \times \G} \one_{ \{ |y-x|<r \} } (x) \left( d(y-x,P) \right)^2  \, dV (x,P) \right| \, d\V(y) \, dr \label{energy_difference_3}
\end{align}
And as in Step $1$ in $(\ref{discrete_varifolds_lipschitz_case})$, for fixed $y$ and $\alpha<r<1$, we have by definition of $V_i$:
\begin{align}
& \left| \int_{\R^n \times \G} \one_{ \{ |y-x|<r \} } (x) \left( d(y-x,P) \right)^2  \, dV_i (x,P) - \int_{\R^n \times \G} \one_{ \{ |y-x^\prime |<r \} } (x^\prime) \left( d(y-x^\prime,P^\prime) \right)^2  \, dV (x^\prime,P^\prime) \right| \nonumber \\
& \leq \sum_{K \in \cK_i} \int_{x \in K} \int_{ K \times \G} \left| \one_{ B_r(y) } (x) \left( d(y-x,P_K^i) \right)^2  - \one_{ B_r(y) } (x^\prime) \left( d(y-x^\prime,P^\prime) \right)^2 \right|  \, dV(x^\prime,P^\prime) \, \frac{d \cL^n(x)}{|K|} \label{discrete_varifolds_lipschitz_case_2}
\end{align}
And in $(\ref{discrete_varifolds_lipschitz_case_2})$, either $x,x^\prime \in B_r(y)$ and in this case
\begin{align*}
\left| \one_{ B_r(y) } (x) \left( d(y-x,P_K^i) \right)^2  - \one_{ B_r(y) } (x^\prime) \left( d(y-x^\prime,P^\prime) \right)^2 \right| & \leq 2 r \, \left|  d(y-x,P_K^i) - d(y-x^\prime,P^\prime) \right| \\
& \leq 2r \, \left( |x-x^\prime | + |y - x^\prime| \| P_K^i - P^\prime \| \right) \\
& \leq 2\, \left( |x-x^\prime | + \| P_K^i - P^\prime \| \right) \: ,
\end{align*}
either $\left\lbrace \begin{array}{l}
x \in B_r(y) \text{ and } x^\prime \notin B_r(y) \text{ or,}\\
x^\prime \in B_r(y) \text{ and } x \notin B_r(y),
\end{array} \right.$ and in this case
\[
\left| \one_{ B_r(y) } (x) \left( d(y-x,P_K^i) \right)^2  - \one_{ B_r(y) } (x^\prime) \left( d(y-x^\prime,P^\prime) \right)^2 \right| \leq r^2 \leq 1 \: .
\]
Notice that, as $| x - x^\prime | \leq \delta_i$ this second case can only happen for $x, \, x^\prime \in B_{r + \delta_i}(y) \setminus B_{r-\delta_i}(y)$. Consequently,
\begin{align*}
& \left| \int_{\R^n \times \G} \one_{ \{ |y-x|<r \} } (x) \left( d(y-x,P) \right)^2  \, dV_i (x,P) - \int_{\R^n \times \G} \one_{ \{ |y-x^\prime |<r \} } (x^\prime) \left( d(y-x^\prime,P^\prime) \right)^2  \, dV (x^\prime,P^\prime) \right| \nonumber \\
& \leq \sum_{K \in \cK_i} \int_{x \in K} \int_{ K \times \G}  2\, \left( |x-x^\prime | + \| P_K^i - P^\prime \| \right) \, dV(x^\prime,P^\prime) \, \frac{d \cL^n(x)}{|K|} \\
& + \sum_{K \in \cK_i} r^2 \V \left( K \cap B_{r + \delta_i}(y) \setminus B_{r-\delta_i}(y) \right)  \underbrace{ \frac{\left|   K \cap B_{r + \delta_i}(y) \setminus B_{r-\delta_i}(y) \right|}{|K|} }_{\leq 1} \\
& \leq 2\delta_i \, \V(\R^n) + \int_{\R^n \times \G} \| P^i (x^\prime) - P^\prime \| \, dV(x^\prime,P^\prime) + \V \left(  B_{r + \delta_i}(y) \setminus B_{r-\delta_i}(y) \right) \\
& \leq 2\left( \delta_i +C \delta_i^\beta \right) \V(\R^n) + \V \left(  B_{r + \delta_i}(y) \setminus B_{r-\delta_i}(y) \right) \text{ thanks to } (\ref{discrete_varifold_convergence_tangent_plane}) \text{ in Step } 2 \: .
\end{align*}
Notice that
\begin{align*}
\int_{r=0}^1 \V \left(  B_{r + \delta_i}(y) \setminus B_{r-\delta_i}(y) \right) \, dr&  = \int_{r=0}^1 \V \left(  B_{r + \delta_i}(y) \right) \, dr - \int_{r=\delta_i}^1 \V \left(   B_{r-\delta_i}(y) \right) \, dr \\
& =  \int_{r=\delta_i}^{1 + \delta_i} \V \left(  B_r(y) \right) \, dr -  \int_{r=0}^{1-\delta_i} \V \left(  B_r(y) \right) \, dr \\
& \leq \int_{r = 1-\delta_i}^{1+\delta_i} \V(B_r(y)) \, dr \leq 2 \delta_i \V(\R^n) \: .
\end{align*}
Eventually, by $(\ref{energy_difference_3})$,
\begin{align*}
\left| \int_{\R^n \times \G} \right. & \left. E_\alpha (x,P,V) \, d V_i (x,P) - \int_{\R^n \times \G} E_\alpha (x,P,V) \, d V(x,P) \right| \\
& \leq \frac{1}{\alpha^{d+3}} \int_{r=0}^1 \int_{\R^n} 2\left( \delta_i +C \delta_i^\beta \right) \V(\R^n) + \V \left(  B_{r + \delta_i}(y) \setminus B_{r-\delta_i}(y) \right) \, d\V(y) \, dr \\
& \leq \frac{1}{\alpha^{d+3}}  \left[ 2\left( \delta_i +C \delta_i^\beta \right) \V(\R^n)^2 + \int_{\R^n} \int_{r=0}^1 \V \left(  B_{r + \delta_i}(y) \setminus B_{r-\delta_i}(y) \right) \, dr  \, d\V(y) \right] \\
& \leq \frac{1}{\alpha^{d+3}}  \V(\R^n)^2 \left[ 4 \delta_i + 2C \delta_i^\beta \right] \: .
\end{align*}

\item Step $5$: It remains to estimate $(\ref{energy_difference_2})$, we prove that 
\begin{equation} \label{energy_difference_5}
\left| \int_{\R^n \times \G} E_\alpha (x,P,V) \, d V_i (x,P) - \int_{\R^n \times \G} E_\alpha (x,P,V_i) \, d V_i(x,P) \right| \leq \: \frac{1}{\alpha^{d+3}} 4 \V(\R^n)^2 \delta_i .
\end{equation}
Indeed, exactly as previously (but fixing $x$ and integrating against $\|V_i\|$, $\V$ instead of $V_i$, $V$, so that the term depending on $P^i$ does not take part into this estimate), we have
\begin{align}
\left| E_\alpha (x,P,V_i) \right. & \left. - E_\alpha (x,P,V) \right| \nonumber \\
& \leq \frac{1}{\alpha^{d+3}} \int_{r=0}^1 \left| \int_{B_r(x)} d(y-x,P)^2 \, d \|V_i\|(y) - \int_{B_r(x)} d(y^\prime-x,P)^2 \, d \V(y^\prime) \right| \, dr \nonumber \\
& \leq  \frac{1}{\alpha^{d+3}} \int_{r=0}^1 \left(  2 \delta_i \V(\R^n) +  \V \left( B_{r + \delta_i}(y) \setminus B_{r-\delta_i}(y)\right) \right) \, dr \nonumber \\
& \leq \frac{1}{\alpha^{d+3}} \V(\R^n) 4 \delta_i \: . \label{discrete_varifold_case}
\end{align}
We conclude this step by integrating against $V_i$, reminding that $V_i (\R^n \times \G) = \|V_i\| (\R^n) = \V(\R^n)$.

\item Step $6$: By $(\ref{energy_difference_4})$ and $(\ref{energy_difference_5})$,
\begin{align}
\left| \int_{\R^n \times \G} E_{\alpha_i} (x,P,V) \, dV(x,P) - \int_{\R^n \times \G} E_{\alpha_i} (x,P,V_i) \, dV_i(x,P) \right| & \leq \frac{1}{\alpha_i^{d+3}} \V(\R^n)^2  \left( 8 \delta_i + 2C \delta_i^\beta \right) \label{energy_difference_6} \\
& \xrightarrow[i \to +\infty]{} 0 \nonumber
\end{align}
thanks to $(\ref{quantitative_condition_on_scale})$. Then, by monotone convergence and $(\ref{energy_difference_6})$,
\begin{align*}
\int_{\R^n \times \G} E_0 (x,P,V) \, dV(x,P) & = \lim_{i \to + \infty} \int_{\R^n \times \G} E_{\alpha_i} (x,P,V) \, dV(x,P)\\
& = \lim_{i \to + \infty} \int_{\R^n \times \G} E_{\alpha_i} (x,P,V_i) \, dV_i (x,P) \: .
\end{align*}

\end{enumerate}
\end{proof}

\begin{appendices}
\section{Appendix: The approximate averaged height excess energy as a tangent plane estimator}

Throughout this section, $(V_i)_i$ is a sequence of $d$--varifolds weakly--$\ast$ converging to some $d$--varifold $V$ and $(\alpha_i)_i$ is a decreasing sequence of positive numbers tending to $0$ and such that
\begin{equation} \label{alpha_choice}
\sup_{\substack{ x \in \overline{\omega} \\  P \in \G }} \left| E_{\alpha_i}^\omega (x,P,V_i) - E_{\alpha_i}^\omega (x,P,V) \right| \: .
\end{equation}
The existence of such a sequence of $(\alpha_i)_i$ is given by Proposition~\ref{alpha_i_choice} in general, and in the case of discrete varifolds associated to a varifold $V$, \eqref{alpha_choice} holds as soon as
\[
\frac{\delta_i}{\alpha_i^{d+3}} \xrightarrow[i \to +\infty]{} 0 \text{ thanks to } (\ref{discrete_varifold_case}) \: .
\]
We want to show that under this condition on the choice of $(\alpha_i)_i$, for fixed $x \in \Omega$, the minimizers of $P \mapsto E_{\alpha_i}^x (P) = E_{\alpha_i}^\omega (x,P,V_i)$ converge, up to some subsequence, to minimizers of $P \mapsto E_0^x (P) = E_0^\omega (x,P,V)$. In the proofs, we shorten $E_{\alpha_i}^x (P) = E_{\alpha_i}^\omega (x,P,V_i)$ and $E_0^x (P) = E_0^\omega (x,P,V)$. We begin with studying the pointwise approximate averaged height excess energy with respect to $P \in G$, for fixed $x \in \Omega$ and for a fixed $d$--varifold $V$.

\subsection{The pointwise approximate averaged height excess energy}

We now fix a $d$--varifold (not supposed rectifiable nor with bounded first variation) in some open set $\Omega \subset \R^n$ and we study the continuity of $E_\alpha (x,P,V)$ with respect to $P \in \G$ and then $x \in \Omega$.

\begin{prop} \label{E_i_x_continuity}
Let $0< \alpha < 1$. Let $V $ be a $d$--varifold in an open set $\Omega \subset \R^n$ such that $\V (\Omega) < +\infty$. Then, for every $P$, $Q \in \G$,
\[
\left| E_\alpha (x,P,V) - E_\alpha (x,Q,V) \right| \leq 2 \| P-Q \| \, \int_{r = \alpha}^1 \frac{1}{r^{d+1}} \V (B_r (x) ) \, dr
\]
In particular, $P \mapsto E_\alpha (x,P,V)$ is Lipschitz with constant $K_\alpha \leq \displaystyle \frac{2}{\alpha^{d+1}}  \V (\Omega )$.
If in addition $\forall \alpha < r < 1, \: \V (B_r(x)) \leq C r^d$ then $K_\alpha \leq C \V (\Omega) \ln \frac{1}{\alpha}$.
\end{prop}

\begin{proof}
Let $P$, $Q \in \G$ then,
\[
\left| E_\alpha (x,P,V) - E_\alpha (x,Q,V) \right| \leq \int_{r = \alpha}^1 \frac{1}{r^{d+1}} \int_{B_r (x)} \left| \left( \frac{d (y-x,P)}{r} \right)^2 - \left( \frac{d (y-x,Q)}{r} \right)^2 \right| \, d \V(y) \, dr \: .
\]
If $\pi_P$ (respectively $\pi_Q$) denotes the orthogonal projection onto $P$ (respectively $Q$), recall that $| d(y-x,P) - d(y-x,Q) | \leq \| P -Q \| |y-x|$. Indeed
\begin{align*}
d(y-x,P) & = \left| y-x - \pi_P (y-x) \right| \\
& \leq \left| y-x - \pi_Q (y-x) \right| + | \pi_Q (y-x) - \pi_P (y-x) | \\
& \leq d(y-x , Q) + \underbrace{\| \pi_Q - \pi_P \|_{op} }_{= \| P-Q\| \text{ by definition}} |y-x| \: .
\end{align*}
Moreover $y \in B_r(x)$ so that $\displaystyle \frac{d(y-x,P)}{r} \leq 1$ and thus
\begin{align*}
\left| \left( \frac{d (y-x,P)}{r} \right)^2 - \left( \frac{d (y-x,Q)}{r} \right)^2 \right| & \leq 2 \, \left| \frac{d (y-x,P)}{r} -  \frac{d (y-x,Q)}{r}  \right| \\
& \leq 2 \| P -Q \| \, \frac{|y-x|}{r} \leq 2 \| P -Q \| \: .
\end{align*}
Consequently,
\[
\left| E_\alpha (x,P,V) - E_\alpha (x,Q,V) \right| \leq 2 \| P-Q \| \, \int_{r = \alpha}^1 \frac{1}{r^{d+1}} \V (B_r (x) ) \, dr \: .
\]
\end{proof}

We now study the continuity of $x \mapsto E_\alpha (x,P,V)$.

\begin{prop} \label{spatial_regularity}
Let $0< \alpha < 1$. Let $V $ be a $d$--varifold in an open set $\Omega \subset \R^n$ such that $\V (\Omega) < +\infty$. Then,
\[
\sup_{P \in \G} | E_\alpha (x,P,V) - E_\alpha (z,P,V) | \xrightarrow[z \to x]{} 0 \: .
\]
\end{prop}

\begin{proof}
First notice that for all $x, \, y, \, z \in \Omega$ and $P \in \G$,
\[
|d(y-x,P) - d(y-z,P)| = \big| |y-x - \pi_P (y-x)| - |y-z - \pi_P (y-z)| \big| \leq | z-x - \pi_P (z-x)| = d(z-x,P) \: .
\]
We now split $B_r(x) \cup B_r (z)$ into $\left( B_r (x) \cap B_r (z) \right)$ and $\left( B_r(x) \triangle B_r(z) \right)$ so that
\begin{align}
\left| \int_{B_r (x)} \right. & \left. \left( \frac{d(y-x,P)}{r} \right)^2 \, d \V(y) - \int_{B_r(z)} \left( \frac{d(y-z,P)}{r} \right)^2 \, d \V(y) \right| \nonumber \\
& \leq \int_{B_r (x) \cap B_r(z)} \left| \left( \frac{d(y-x,P)}{r} \right)^2 - \left( \frac{d(y-z,P)}{r} \right)^2 \right| \, d \V(y) \label{first_integral} \\
& + \int_{B_r(x) \setminus B_r (z)} \left(\frac{d(y-x,P)}{r} \right)^2 \, d \V(y) + \int_{B_r(z) \setminus B_r (x)} \left(\frac{d(y-z,P)}{r} \right)^2 \, d \V(y) \: . \label{two_other_integrals}
\end{align}
We use the estimate linking $d(y-x,P)$ and $d(y-z,P)$ to control the first integral and then we show that the two other terms tend to $0$.

\noindent Concerning the first integral $(\ref{first_integral})$:
\begin{align*}
& \int_{B_r (x) \cap B_r(z)} \left| \left( \frac{d(y-x,P)}{r} \right)^2 - \left( \frac{d(y-z,P)}{r} \right)^2 \right| \, d \V(y)\\
 \leq & \int_{B_r (x) \cap B_r(z)} 2 \left|  \frac{d(y-x,P)}{r} - \frac{d(y-z,P)}{r} \right| \, d \V(y) \\
 \leq & 2 \frac{|z-x|}{r} \V \left( B_r(x) \cap B_r(z)\right) \: .
\end{align*}
Concerning the two other integrals $(\ref{two_other_integrals})$:
\begin{align*}
\int_{B_r(x) \setminus B_r (z)} & \left( \frac{d(y-x,P)}{r} \right)^2 \, d \V(y) + \int_{B_r(z) \setminus B_r (x)} \left( \frac{d(y-z,P)}{r} \right)^2 \, d \V(y) \\
& \leq \V \left( B_r(x) \triangle B_r(z) \right) \leq R(x,r-|z-x|,r+|z-x|) \: ,
\end{align*}
where $\displaystyle R(x,r_{\text{min}},r_{\text{max}}) := \left\lbrace y \in \Omega \, | \, r_{\text{min}} \leq |y-x| \leq r_{\text{max}} \right\rbrace$.

\noindent Therefore,
\begin{align*}
& | E_\alpha (x,P,V) - E_\alpha (z,P,V) | \\
& \leq 2 |z-x| \int_{r = \alpha}^1 \V(B_r(x) \cap B_r(z)) \frac{dr}{r^{d+2}} + \int_{r = \alpha}^1 \V \left( B_r(x) \triangle B_r(z) \right) \frac{dr}{r^{d+1}} \\
& \leq \frac{2}{d+1} |z-x| \frac{1}{\alpha^{d+1}} \V(\Omega) + \frac{1}{\alpha^{d+1}} \int_{r = 0}^1 \V\left( R(x,r-|z-x|,r+|z-x|) \right) \, dr \: .
\end{align*}
The second term tends to $0$ when $|z-x| \rightarrow 0$, by dominated convergence, since \[
\lim_{z \to x} \V\left( R(x,r-|z-x|,r+|z-x|) \right) = \V (\partial B_r(x)) \: .
\]
\end{proof}

\subsection{$\Gamma$--convergence of $P \mapsto E_{\alpha_i}^\omega (x,P,V_i)$ to $P \mapsto E_0^\omega (x,P,V)$.}
 
\begin{prop} \label{liminf}
Let $\Omega \subset \R^n$ be an open set and let $\omega \subset\subset \Omega$ be a relatively compact open subset such that $\V(\partial \omega) = 0$. Let $(V_i )_i$ be a sequence of $d$--varifolds weakly--$\ast$ converging to $V $. Assume that $(\alpha_i)_i$ are chosen as explained in (\ref{alpha_choice}), uniformly in $\omega$. For $(S_i)_i \subset \G$ such that $S_i \xrightarrow[i \infty]{} S$ then, for all $x \in \omega$,
\[
\lim_{i \to \infty} E_{\alpha_i}^\omega  (x,S,V_i) = E_0^\omega (x,S,V) \leq \liminf_{i \to \infty} E_{\alpha_i}^\omega  (x,S_i,V_i) \: .
\]
\end{prop}

\begin{proof}
By monotone convergence, we already know that 
\begin{equation} \label{monotone_convergence}
E_0^\omega (x ,S,V) = \lim_{i \to \infty} E_{\alpha_i}^\omega (x,S,V) \: .
\end{equation}
So we now want to estimate $\left| E_{\alpha_i}^\omega (x,S,V) - E_{\alpha_i}^\omega (x, S_i,V) \right|$. Let us start with extracting some $(S_{\phi (i)} )_i$ such that
\[
\|S_{\phi (i)} - S\| \frac{1}{\alpha_i^{d+1}} \xrightarrow[i \to \infty]{} 0
\]
so that we can now apply the regularity property (Proposition~\ref{E_i_x_continuity}) of $E_\alpha (x,P,V)$ with respect to $P$:
\[
\left| E_{\alpha_i}^\omega (x,S,V) - E_{\alpha_i}^\omega (x, S_{\phi(i)},V) \right| \leq \frac{2}{\alpha_i^{d+1}} \V (\omega) \| S - S_{\phi (i)} \| \xrightarrow[i \to \infty]{} 0 .
\]
thus
\begin{equation} \label{convergence_2}
E_0^\omega (x ,S,V) = \lim_{i \to \infty} E_{\alpha_i}^\omega (x, S_{\phi(i)},V) .
\end{equation}
Notice that $\phi$ only depends on $(\alpha_i)_i$.

\noindent As the sequence $(\alpha_i)_i$ is decreasing, $\alpha_{\phi(i)} \leq \alpha_i$ and then $E_{\alpha_i}^\omega (x, Q,V) \leq E_{\alpha_{\phi (i)}}^\omega (x, Q,V)$ for all $Q \in \G$, which implies in particular that
\begin{equation} \label{convergence_3}
\lim_{i \to \infty} E_{\alpha_i}^\omega (x, S_{\phi(i)},V) \leq \liminf_{i \to \infty} E_{\alpha_{\phi(i)}}^\omega (x, S_{\phi(i)},V) \: .
\end{equation}
We now apply the uniform convergence of $\left| E_{\alpha_i}^\omega (\cdot, \cdot, V) - E_{\alpha_i}^\omega (\cdot, \cdot, V_i) \right|$ \eqref{alpha_choice},
\begin{equation} \label{convergence_4}
\left| E_{\alpha_{\phi(i)}}^\omega (x, S_{\phi(i)}, V) - E_{\alpha_{\phi(i)}}^\omega (x, S_{\phi(i)}, V_{\phi(i)})  \right| \xrightarrow[i \to \infty]{} 0 \: ,
\end{equation}
so that by \eqref{convergence_2}, \eqref{convergence_3} and \eqref{convergence_4}
\begin{equation} \label{convergence_5}
E_0^\omega (x ,S,V) \leq \liminf_{i \to \infty} E_{\alpha_{\phi(i)}}^\omega (x, S_{\phi(i)}, V) = \liminf_{i \to \infty} E_{\alpha_{\phi(i)}}^\omega (x, S_{\phi(i)}, V_{\phi(i)}) \: .
\end{equation}
As $\liminf_i E_{\alpha_i}^\omega (x, S_i, V_i) = \lim_i E_{\alpha_{\theta(i)}}^\omega (x, S_{\theta(i)}, V_{\theta(i)})$ for some extraction $\theta$, we now apply \eqref{convergence_5} to these extracted sequences $(S_{\theta(i)})_i$ and $(V_{\theta(i)})_i$ so that there exists an extraction $\phi$ such that
\begin{align*}
E_0^\omega (x ,S,V) & \leq \liminf_{i \to \infty} E_{\alpha_{\theta(\phi(i))}}^\omega (x, S_{\theta(\phi(i))}, V_{\theta(\phi(i))}) \\
& = \lim_i E_{\alpha_{\theta(i)}}^\omega (x, S_{\theta(i)}, V_{\theta(i)}) \text{ since the whole sequence } E_{\alpha_{\theta(i)}}^\omega (x, S_{\theta(i)}, V_{\theta(i)}) \text{ converges} \\
& =  \liminf_i E_{\alpha_i}^\omega (x, S_i, V_i) \: .
\end{align*}

\end{proof}

\noindent We now turn to the consequences of this $\Gamma$--convergence property on the minimizers.

\begin{prop}
Let $V_i $ be a sequence of $d$--varifolds weakly--$\ast$ converging to $V $ in some open set $\Omega \subset \R^n$ and assume that $(\alpha_i)_i$ are chosen as explained in (\ref{alpha_choice}), uniformly in $\omega \subset\subset \Omega$ open subset such that $\V(\partial \omega) = 0$. For $x \in \omega$ and $i \in \N$, let $T_i (x) \in \argmin_{P \in \G} E_{\alpha_i}^\omega (x, P, V_i)$. Then,
\begin{enumerate}
\item Any converging subsequence of $(T_i(x))_i$ tends to a minimizer of $E_0^\omega (x, \cdot , V)$.
\item  $\displaystyle \min_{P \in \G} E_{\alpha_i}^\omega (x, P, V_i) \xrightarrow[i \to\infty]{} \min_{P \in \G} E_0^\omega (x, P, V)$.
\item If $V$ is an integral rectifiable $d$--varifold with bounded first variation then
\[
\argmin_{P \in \G} E_0^\omega (x, P, V) = \left\{ T_x M \right\} \: , 
\]
hence for $\V$--almost every $x$, $\displaystyle T_i (x) \xrightarrow[i \to \infty]{} T_x M$.
\end{enumerate} 
\end{prop}


\begin{proof}
First, for fixed $x$ and $i$, $P \mapsto E_{\alpha_i}^\omega (x, P, V_i)$ is continuous and $\G$ is compact so that  ${ \argmin_{P \in \G} E_{\alpha_i}^\omega (x, P, V_i) \neq \emptyset }$. Let $T_i(x) \in \argmin_{P \in \G} E_{\alpha_i}^\omega (x, P, V_i)$ be a sequence of minimizers, as $\G$ is compact, one can extract a subsequence converging to some $T_\infty (x)$. Now applying the previous result (Proposition~\ref{liminf}), we get for every $P \in \G$,
\begin{align*}
E_0^\omega (x, T_\infty (x), V) & \leq \liminf_{i\to \infty} E_{\alpha_i}^\omega (x, T_i (x), V_i) \\
& \leq \limsup_{i\to \infty} E_{\alpha_i}^\omega (x, T_i (x), V_i)\\
& \leq \limsup_{i\to \infty} E_{\alpha_i}^\omega (x, P, V_i) \\
& = \lim_i E_{\alpha_i}^\omega (x, P, V_i) = E_0^\omega (x, P, V)  \\
& \leq E_0^\omega (x, T_\infty (x), V) \text{ for } P=T_\infty (x) \: .
\end{align*}
Therefore $T_\infty(x)$ minimizes $E_0^\omega (x, \cdot, V)$ which allows to conclude that the limit of any subsequence of minimizers of $E_{\alpha_i}^\omega (x, \cdot , V_i)$ is a minimizer of $E_0^\omega (x, \cdot, V)$. It also proves that 
\[
\lim_i \min_{P \in \G} E_{\alpha_i}^\omega (x, P, V_i) = \lim_i E_{\alpha_i}^\omega (x, T_i (x), V_i) = E_0^\omega (x, T_\infty (x), V) = \min_{P \in \G} E_0^\omega (x, P, V) \: .
\]
Assume now that $E_0^\omega (x, \cdot , V)$ admits a unique minimizer $T(x)$. We have just shown that every subsequence of $(T_i(x))_i$ converges to $T(x)$. As $\G$ is compact, it is enough to show that the whole sequence is converging to $T(x)$. Now if $V$ is an integral $d$--rectifiable varifold with bounded first variation, for $\V$--almost every $x$, $T_x M$ is the unique minimizer of $E_0^\omega (x, \cdot , V)$ (see Prop. \ref{min_of_continuous_energy}) so that for $\V$--almost every $x \in \omega$,
\[
T_i (x) \xrightarrow[i \to \infty]{} T_x M \: .
\]
\end{proof}

\begin{remk}
Since $E_0^\omega (x, \cdot, V)$ has no continuity property, the existence of a minimizer of $E_0^\omega (x, \cdot, V)$ is not clear a priori. However, as $\G$ is compact, every sequence of minimizers $(T_i(x))_i$ admits a converging subsequence so that $\argmin_{P \in \G} E_0^\omega (x, P, V)$ is not empty.
\end{remk}

We end with studying the continuity of the minimum $\min_{P \in \G} E_{\alpha_i} (x, P,V_i)$ with respect to $x$ (for fixed $i$ and $V_i$).

\begin{prop}
Assume that $V_i$ weakly--$\ast$ converges to $ V$ in some open set $\Omega \subset \R^n$ and let $(\alpha_i)_i > 0$. Then for every fixed $i$ and $\omega \subset\subset \Omega$, the function $\displaystyle x \mapsto \min_{P \in \G} E_{\alpha_i}^\omega (x,P,V_i)$ is continuous  in $\omega$.

\noindent Moreover, every converging sequence of minimizers $\left( T_i (z_k) \in \argmin_P E_{\alpha_i}^\omega (z_k,P,V_i) \right)_k$ tends to a minimizer of $E_{\alpha_i}^\omega (x,\cdot,V_i)$ when $z_k \rightarrow x$ and for a fixed $i$.
\end{prop}

\begin{remk}
As $i$ is fixed, meaning actually that a scale $\alpha = \alpha_i > 0$ and a $d$--varifold $V = V_i$ are fixed, we keep the notations $V_i$ and $\alpha_i$, with the explicit index $i$, only to be coherent with the whole context of this section and with the notations of the previous results. But that is why we do not assume anything on the choice of $\alpha_i>0$ and $\omega \subset\subset \Omega$.
\end{remk}

\begin{proof}
Let $i$ be fixed. First we show that if $(z_k)_k \subset \omega$ is such that 
\[ \left\lbrace \begin{array}{l}
\left| z_k - x \right| \xrightarrow[k \to \infty]{} 0 \\
T_i (z_k) \xrightarrow[k \to \infty]{} T_i^\infty \text{ where } T_i(z_k) \in \argmin_P E_{\alpha_i}^\omega (z_k,P,V_i) \: ,
\end{array} \right. \]
then,
\begin{equation} \left\lbrace \begin{array}{l}
T_i^\infty \in \argmin_P E_{\alpha_i}^\omega (x,P,V_i) \text{ and }\\
\min_P E_{\alpha_i}^\omega (z_k,P,V_i) = E_{\alpha_i}^\omega (z_k , T_i(z_k) , V_i) \xrightarrow[k \to \infty]{} E_{\alpha_i}^\omega (x,T_i^\infty,V_i) = \min_P E_{\alpha_i}^\omega (x,P,V_i) \: . \label{convergence_tangent_plane_estimator_1}
\end{array} \right.
\end{equation}
Indeed,
\begin{align*}
\left| E_{\alpha_i}^\omega (x , T_i^\infty , V_i) \right. & - \left. E_{\alpha_i}^\omega (z_k, T_i(z_k), V_i) \right| \\
\leq & \left| E_{\alpha_i}^\omega (x, T_i^\infty , V_i) - E_{\alpha_i}^\omega (x,T_i(z_k),V_i) \right| + \left| E_{\alpha_i}^\omega (x , T_i(z_k) , V_i) - E_{\alpha_i}^\omega (z_k, T_i(z_k) , V_i) \right| \\
\leq & K(\alpha_i) \| T_i^\infty - T_i(z_k) \| + \sup_P \left| E_{\alpha_i}^\omega (x,P,V_i) - E_{\alpha_i}^\omega (z_k,P,V_i) \right|
\end{align*}
applying Proposition~\ref{E_i_x_continuity} to the first term, $K(\alpha_i)$ is a constant depending only on $\alpha_i$. Moreover, by Proposition~\ref{spatial_regularity}, the second term tends to zero when $k$ goes to $\infty$. Consequently,
\[
E_{\alpha_i} (x, T_i^\infty, V_i) = \lim_{k \to \infty} E_{\alpha_i} (z_k , T_i(z_k) , V_i) \: .
\]
And for every $P \in \G$,
\begin{align*}
E_{\alpha_i}^\omega (x , T_i^\infty, V_i) & = \lim_{k \to \infty} E_{\alpha_i}^\omega (z_k , T_i(z_k) , V_i) \\
&  \leq \lim_{k \to \infty} E_{\alpha_i}^\omega (z_k , P , V_i)\\
& = E_{\alpha_i}^\omega (x , P ,V_i) \rm{\: by \: Proposition~\ref{spatial_regularity}},
\end{align*}
which yields \eqref{convergence_tangent_plane_estimator_1}.

\noindent It remains to prove the continuity of $\displaystyle x \mapsto \min_{P \in \G} E_{\alpha_i}^\omega (x,P,V_i)$. Let $x$ and $(z_k)_k \in \omega$ be such that $z_k \xrightarrow[k \to \infty]{} x$ and consider a subsequence $(z_{\phi(k)})_k$ such that 
\begin{equation} \label{convergence_tangent_plane_estimator_2}
\limsup_k E_{\alpha_i}^\omega (z_k , T_i(z_k) , V_i) = \lim_k E_{\alpha_i}^\omega (z_{\phi(k)} , T_i(z_{\phi(k)}) , V_i) \: .
\end{equation} 
As $\G$ is compact, there exists an extraction $\theta$ such that $(T_i(z_{\phi(\theta(k))}))_k$ is converging and then applying the previous argument \eqref{convergence_tangent_plane_estimator_1} to $(z_{\phi(\theta(k))})_k$ and $(T_i(z_{\phi(\theta(k))}))_k$,
\begin{equation} \label{convergence_tangent_plane_estimator_3}
\lim_{k \to + \infty} E_{\alpha_i}^\omega \left( z_{\phi(\theta(k))} , T_i(z_{\phi(\theta(k))}) , V_i \right) = \min_{P \in \G} E_{\alpha_i}^\omega (x , P , V_i) \: .
\end{equation}
Eventually, by \eqref{convergence_tangent_plane_estimator_2} and \eqref{convergence_tangent_plane_estimator_3},
\[
\limsup_{k \to +\infty} E_{\alpha_i}^\omega (z_k , T_i(z_k) , V_i ) =  \min_{P \in \G} E_{\alpha_i}^\omega (x , P , V_i) \: .
\]
Similarly $\displaystyle \liminf_k E_{\alpha_i}^\omega (z_k , T_i(z_k) , V_i) =  \min_P E_{\alpha_i}^\omega (x , P , V_i)$ which concludes the proof of the continuity.
\end{proof}

\end{appendices}

\textbf{Acknowledgements.} I would like to thank my PhD advisors Gian Paolo Leonardi and Simon Masnou for their constant support and encouragement.

\nocite{*}
\bibliographystyle{plain}
\bibliography{biblio_rectifiability}

\end{document}